\documentclass[11pt]{article}
\usepackage{amsmath}
\usepackage{lscape,tikz}

\definecolor{darkblue}{rgb}{0.2,0.2,0.71}
\usepackage{todonotes}
\usepackage{epsfig,cancel,ulem, amsthm,amssymb}
\usepackage{color,soul}
\usetikzlibrary{decorations.markings,arrows}

\usepackage{wrapfig}
\usepackage{color}
\usepackage{graphicx,framed}
\definecolor{shadecolor}{rgb}{0.95, 0.95, 0.86}
\definecolor{darkgreen}{rgb}{0.2, 0.5,  0}

\def\&{\vspace{-5pt}&}

\def\Tr{ {\rm Tr}}

\def \eqref#1{(\ref{#1})}
\def \& {&\hspace{-10pt}}

\newcommand{\bt}{\beta} 
 
\newcommand{\ga}{\gamma} 
 
\renewcommand{\O}{\Omega}

\newcommand{\e}{\epsilon}

\newcommand{\br}{{\mathbb R}}

\newcommand{\gr}{\mathfrak{g}} 
 
\newcommand{\g}{\mathfrak{gl}_r}

\newcommand{\Y}{\Lambda}

    \newcommand{\B}{\mathcal{B}}
    \newcommand{\V}{\mathcal{V}}
     \newcommand{\lop}[1]{\mathfrak{L}(#1)}
      \newcommand{\Lie}{\mathrm{Lie}}
\newcommand{\bil}[2]{{\mathrm{Tr}\left( #1 \circ #2\right)}}

\newcommand{\nneg}{\mathfrak{n}}

  \newcommand{\bneg}{\mathfrak{b}}

\newcommand{\f}{\mathcal{F}}

\newcommand{\Q}{\mathcal{Q}}

\newtheorem{theorem}{Theorem}[section]
\newtheorem{example}[theorem]{Example}
\newtheorem{exercise}[theorem]{Exercise}

\newtheorem{lemma}[theorem]{Lemma}
\newtheorem{remark}[theorem]{Remark}

\newtheorem{proposition}[theorem]{Proposition} 
\newtheorem{corollary}[theorem]{Corollary} 
\newtheorem{definition}[theorem]{Definition}

\newtheorem{scheme}[theorem]{Scheme}

\def\V {\mathcal V}

\def\bt{\begin{theorem}}
\def\et{\end{theorem}}
\def\bc{\begin{corollary}}
\def\ec{\end{corollary}}
\def\bx{\begin{example}}
\def\ex{\end{example}}
\def\bxr{\begin{exercise}\small}
\def\exr{\end{exercise}}
\def\bl{\begin{lemma}}
\def\el{\end{lemma}}
\def\bd{\begin{definition}}
\def\ed{\end{definition}}
\def\bp{\begin{proposition}}
\def\ep{\end{proposition}}

\def\br{\begin{remark}}
\def\er{\end{remark}}

\def\be{\begin{equation}}
\def\ee{\end{equation}}
\def\bsc{\begin{scheme}}
\def\esc{\end{scheme}}

\def\beq{\begin{equation*}}
\def\eeq{\end{equation*}}
\def\&{\hspace{-15pt}&}
\def\bea{\begin{eqnarray}}
\def\eea{\end{eqnarray}}

\def\L{\mathcal L}

\def\T{{\mathcal T}}

\def\1{{\bf 1}}

\def\w{{\mathbf w}}

\newcommand{\h}{\mathfrak{h}}
\newcommand{\sll}{\mathfrak{sl}}

\newcommand{\ad}{\mathrm{ad}}

\makeatletter
\@addtoreset{equation}{section}
\makeatother

\begin{document}
\title{From Adler-Gelfand-Dickey Brackets to Logarithmic Dubrovin-Frobenius manifolds}
\author{Yassir Ibrahim Dinar}
\date{}
\maketitle
\begin{abstract}
We construct a new local Poisson bracket compatible with the second unconstrained  Adler-Gelfand-Dickey bracket. The resulting bihamiltonian structure  admits a dispersionless limit and the leading term defines a logarithmic Dubrovin–Frobenius manifold. Furthermore, we show that  this Dubrovin-Frobenius manifold can be constructed on the orbits space of the standard representation of the permutation group.

\end{abstract}
{\small \noindent{\bf Mathematics Subject Classification (2020).} 37K25; 37k30, 53D45, 17B80, 53D17, 13A50, 17B68, 17B08.}


\setcounter{equation}{0}
\setcounter{theorem}{0}
\setcounter{equation}{0}
\setcounter{theorem}{0}

\section{Introduction}

One of the main methods  to obtain examples of Dubrovin-Frobenius manifolds exist within the theory of flat pencils of metrics (equivalently, nondegenerate compatible Poisson brackets of hydrodynamic type). Besides, the leading terms of  a certain type of compatible local  Poisson brackets (a local bihamiltonian structure) which admit(s) a dispersionless limit form a flat pencil of metric \cite{DuRev}.  Moreover, we can obtain compatible local Poisson brackets for any nilpotent element in simple Lie algebras using Drinfeld-Sokolov reduction (see for examples \cite{DS}, \cite{BalFeh1}, \cite{fehercomp},\cite{gDSh2}, \cite{mypaper4}). One of these Poisson brackets is  (or satisfies identities leading to) a classical $W$-algebra.

In \cite{mypaper6}, we developed a uniform construction of algebraic Dubrovin-Frobenius manifolds using Drinfeld-Sokolov bihamiltonian structures associated with what is called distinguished nilpotent elements of semisimple type \cite{Elash}. Then we  analyzed  Drinfeld-Sokolov bihamiltonian structures associated with  subregular nilpotent elements  in the simple Lie algebras \( \sll_3 \) and \( \sll_4 \) in \cite{rank 3}. We  demonstrated  that the bihamiltonian structures fail to define   flat  pencils of metrics using the techniques of \cite{mypaper6}.  However,  starting from  classical \( W \)-algebras, we identified alternative bihamiltonian structures. The leading terms of these new structures define logarithmic Dubrovin-Frobenius manifolds. 

This leads us to study the  second  Adler-Gelfand-Dickey (AGD) bracket $\mathbb B_2^\Q$ \cite{dickeylec}.  It is  defined over the space of differential operators of the form 
\begin{equation}\label{operator_L}
L = D^r + v^1(x) D^{r-1} + v^2(x) D^{r-2} + \dots + v^{r-1}(x) D + v^r(x),~~  D = \frac{d}{dx}.
\end{equation}
 This space is embedded within the ring of pseudodifferential operators of the form 
\[
A = \sum_{k=-\infty}^N a_k(x) D^k.
\]
Where  the residue, differential part and product are given by
\[
\operatorname{res}(A) = a_{-1}, \quad A_+ = \sum_{k=0}^N a_k(x) D^k,~~D^k\circ a_l(x) = \sum_{n\geq 0}\frac{k (k - 1) \cdots (k - n + 1)}{n!} a_l^{(n)}(x).
\]
Then the bracket   of two functionals $F$ and $G$ under \( \mathbb{B}_2^\Q \) is defined  by 
\begin{equation}\label{AGD_bracket}
\{ F, G \}_2^\Q = \int_{S^1} \operatorname{res} \left[ \left( L \dfrac{\delta F}{\delta L} \right)_+ L \dfrac{\delta G}{\delta L} - \left( \dfrac{\delta F}{\delta L} L \right)_+ \dfrac{\delta G}{\delta L} L \right] dx.
\end{equation}
Here, for a functional $F$, we set 
\begin{equation*}\label{variational_derivative}
\dfrac{\delta F}{\delta L} = \sum_{i=1}^r D^{i - r - 1} \left( \dfrac{\delta F}{\delta v^i(x)} \right).
\end{equation*}
Note that the Lie derivative \( \Lie_{\partial_{v^r}} \mathbb{B}_2^\Q \) defines the first Adler-Gelfand-Dickey bracket, and the pair \( (\mathbb{B}_2^\Q, \Lie_{\partial_{v^r}} \mathbb{B}_2^\Q) \) forms a bihamiltonian structure. This  bihamiltonian structure admits a dispersionless limit. However, its leading term   fails to define a flat pencil of metrics (see Proposition \ref{firsagd}). Recall that  its reduction to the space of operators \( L \) with \( v^1(x) = 0 \) yields the constrained Adler-Gelfand-Dickey brackets associated with the $(r-1)$-KdV hierarchies. In this article, however, we focus exclusively on the unconstrained bracket \eqref{AGD_bracket}.  

We then reconsider the construction  of $\mathbb{B}_2^\Q$ via Drinfeld-Sokolov reduction \cite{DS}. It is the classical $W$-algebra associated to regular nilpotent elements in the complex general Lie algebra $\g$ (see section \ref{revDS}). The reduction allows us to express the brackets of $\mathbb{B}_2^\Q$ using different densities than $v^i(x)$ and to use the theory of Lie algebras to analyze the entries of the brackets. 

 To give more details about the main results in this article, let $\e_{i,j}\in \g$ denote the $r \times r$ matrix whose $(\nu, \mu)$-entry is $\delta_{\nu,i}\delta_{\mu,j}$, and fix the following $\mathfrak{sl}_2$-triple $\{L_2, h, f\}$ associated with the regular nilpotent element $L_2$:
\[
L_2 = \sum_{i=1}^{r-1} \e_{i,i+1}, \quad 
h = \frac{1}{2}\sum_{i=1}^{r}(r - 2i + 1)\e_{i,i},\quad
f = \sum_{i=1}^{r-1} i(r - i)\e_{i+1,i}.
\]
We consider the affine loop space 
\[
\mathcal{Q} := L_2 + \lop{\g^f},\quad\text{where}\quad
\g^f := \{g \in \g : [g, f] = 0\}, \quad
\lop{\g^f} := C^\infty(S^1,\g^f).
\]
on which the second Adler-Gelfand-Dickey bracket $\mathbb{B}_2^\Q$ will be defined through Drinfeld-Sokolov reduction. 

Define coordinates \( s^i: Q \to \mathbb{C} \) on $Q := L_2 + \g^f$ by
\[
s^i(g) = \frac{1}{i}\mathrm{Tr}(g^i),\quad (i \neq r),\quad 
s^r(g) = \frac{r - 1}{(r - 1) + \alpha r}\left(\mathrm{Tr}(g^r) + \alpha\, s^1(g)\, s^{r-1}(g)\right),\quad g\in Q.
\]
where $\alpha$ satisfies the quadratic equation
\[
r\alpha^2 + 2(r - 1)\alpha + (r - 1) = 0.
\]
Then the brackets $\{s^i(x), s^j(y)\}_2^\Q$ defined by $\mathbb{B}_2^\Q$ depend at most linearly on the density $s^{r-1}(x)$. Consequently, the Lie derivative 
\[
\mathbb{B}_1^\Q := \mathrm{Lie}_{\partial_{s^{r-1}(x)}}\mathbb{B}_2^\Q
\]
is a local Poisson bracket compatible with $\mathbb{B}_2^\Q$. Furthermore, the resulting  bihamiltonian structure admits a dispersionless limit and the leading term defines a contravariant flat pencil of metrics on $Q$. This flat pencil of metrics gives rise to a logarithmic Dubrovin-Frobenius manifold structure on an open dense subset of $Q$ (see Theorem  \ref{coordins} and Theorem \ref{FPonQ}]).

It is known from \cite{DLZ} that Dubrovin-Frobenius manifold associated with Drinfeld-Sokolov bihamiltonian structure for a regular nilpotent element in simple Lie algebra is locally isomorphic to the polynomial Dubrovin-Frobenius manifold arising on the orbits space of the underlying Weyl group. In this paper, we establish a similar relation by showing that the constructed logarithmic Dubrovin-Frobenius manifold can be realized on the orbits space of the underlying Weyl group  of $\g$,  which is simply the standard representation of the permutation group \( S_r \) (see Theorem \ref{FBonOrb}).

The article is organized as follows. In Sections 2 and 3, we review the theory of local Poisson brackets and Dubrovin-Frobenius manifolds and their relation to flat pencil of metrics. This serves as the foundational framework for the rest of the article. In Section 4, we introduce regular nilpotent elements in the general linear algebra, along with the necessary notations and identities.

Section 5 provides details on the Drinfeld-Sokolov reduction and the construction of the second Adler-Gelfand-Dickey bracket $\mathbb{B}_2^\mathcal{Q}$. The formulations in Section 5 are used in Section 6 to analyze the entries of the brackets of $\mathbb{B}_2^\mathcal{Q}$. In Sections 7 and 8, we investigate the change of these  entries under two successive changes of coordinates: invariant coordinates introduced in Section 7, and further coordinates introduced in Section 8. The coordinates defined in Section 8 enable us to construct a Poisson bracket $\mathbb{B}_1^\mathcal{Q}$ that is compatible with $\mathbb{B}_2^\mathcal{Q}$.

In Section 9, we demonstrate that the leading term of the bihamiltonian structure gives rise to logarithmic Dubrovin-Frobenius manifolds. Here, we use the techniques of \cite{Arsie} and \cite{WuZuo} on constructing similar structures. Additionally, in Section 10, we show that these Dubrovin-Frobenius manifolds can be constructed using invariant theory for the standard representations of permutation groups, using mainly Miura transformation and the coordinates introduced in section 7. Finally, in the last section, we give some remarks about the geometry of the bihamiltonian structure 
$(\mathbb{B}_2^\mathcal{Q},\mathbb{B}_1^\mathcal{Q})$.

Throughout this paper, the base field is the complex numbers \(\mathbb{C}\). Unless otherwise stated, finite-dimensional manifolds are complex manifolds. Smooth maps \(u:S^1\to M\), are differentiated with respect to the real parameter \(x\in S^1\), taking values in the complex manifold \(M\). The Einstein summation convention is used throughout.

\section{Geometry of local Poisson brackets}\label{sec:localPB}

\medskip
Let \(M\) be a manifold with local coordinates \((u^{1},\dots,u^{r})\).
The loop space $\lop M$ of $M$ is defined as the space of smooth functions from the circle $S^1$ to $M$. A local functional on~\(\lop M\)  is an integral of the form
\[
  \mathcal{F}[u]
  =
  \int_{S^{1}}
    F(u(x),u_{x}(x),\dots,u^{(m)}(x))\,dx,
\]
where the integrand \(F\) (called a density) is a holomorphic function of the  variables  \(\partial_{x}^{k}u^i(x)\).   A local Poisson bracket equips the space of  functionals with a Lie algebra structure. In terms of the  densities \(u^{i}(x)\),  it admits the form \cite{DZ}
\begin{equation}\label{eq:genLocPoissBra}
  \{u^{i}(x),u^{j}(y)\}
  \;=\;
  \sum_{k=0}^{N}
    T^{ij}_{k}\!\bigl(u(x),u_{x}(x),\dots,u^{(m)}(x)\bigr)\,
   \delta^{(k)}(x-y),
\end{equation}
for some natural number \(N\). The Dirac delta distribution $ \delta(x-y)$ is defined as
\begin{equation*}\label{eq:deltaFunction}
  \int_{S^{1}} f(y)\,\delta(x-y)\,dy
  \;=\;
  f(x).
\end{equation*}

\bd \cite{fehercomp}
A local Poisson bracket $\{\cdot,\cdot\}$ is called a classical $W$-algebra if there exist coordinates $(z^1,\dots,z^r)$ such that the corresponding brackets have the form
\begin{eqnarray}
\{z^2(x), z^2(y)\}&=& c\, \delta'''(x-y) + 2 z^2(x)\, \delta'(x-y) + z^2_x(x)\, \delta(x-y), \\\nonumber
\{z^2(x), z^j(y)\} &=& j\, z^j(x)\, \delta'(x-y) + (j-1)\, z^j_x(x)\, \delta(x-y), \quad j\neq 2,
\end{eqnarray}
for some nonzero constant $c$. 
\ed 
\bd A pair of local Poisson brackets $\{ \cdot, \cdot \}_1$ and $\{ \cdot, \cdot \}_2$ on $\lop M$ are  compatible or  form a bihamiltonian structure if 
\[
\{ \cdot, \cdot \}_{(\lambda)} := \{ \cdot, \cdot \}_2 + \lambda \{ \cdot, \cdot \}_1
\]
is a local Poisson bracket for any constant $\lambda$. 
\ed

From the formula of the Lie derivative of  local Poisson brackets  given in Example 2.3.1 of \cite{DZ}, we have:

\bc \label{compatible poiss}
For the vector field $X = \partial_{u^k(x)}$, the Lie derivative $\Lie_X \{ \cdot, \cdot \}$ is obtained by differentiating the entries of the Poisson bracket with respect to $u^k(x)$.
\ec

The following proposition provides a method to construct  a bihamiltonian structure.

\bp\label{serg} (Proposition 1 in \cite{serg})
Let $X$ be a vector field on $\lop M$, and let $\{ \cdot, \cdot \}$ be a local Poisson bracket on $\lop M$. If $\Lie_X^2 \{ \cdot, \cdot \} = 0$, then the brackets $\{ \cdot, \cdot \}$ and $\{ \cdot, \cdot \}_1 := \Lie_X \{ \cdot, \cdot \}$ form a bihamiltonian structure.
\ep

Let us fix a local Poisson bracket $\{ \cdot, \cdot \}_2$ on the loop space $\lop M$. Following \cite{DZ}, we assign degree $-1$ to the delta distribution $\delta(x - y)$ and degree $n$ to each derivative $\partial_x^n u^i(x)$. Then the local Poisson bracket $\{ \cdot, \cdot \}_2$ admits an expansion of the form
\begin{eqnarray}\label{loc poiss}
\{u^i(x), u^j(y)\}_2 &=& \sum_{k=-1}^{\infty}\{u^i(x),u^j(y)\}_2^{[k]},\\[5pt]\nonumber
\{u^i(x),u^j(y)\}_2^{[-1]} &=& F_2^{ij}(u(x))\,\delta(x - y),\\[5pt]\nonumber
\{u^i(x),u^j(y)\}_2^{[0]} &=& \Omega_2^{ij}(u(x))\,\delta'(x - y) + \Gamma_{2,k}^{ij}(u(x))\,u_x^k\,\delta(x - y),\\[5pt]\nonumber
\{u^i(x),u^j(y)\}_2^{[k]} &=& S_{2,k}^{ij}(u(x))\,\delta^{(k+1)}(x - y)+\cdots,\quad k>0.
\end{eqnarray}

Here, the densities \(F_2^{ij}(u(x))\), \(\Omega_2^{ij}(u(x))\), \(\Gamma_{2,k}^{ij}(u(x))\), and \(S_{2,k}^{ij}(u(x))\) depend only on  \(u^i(x)\) and not on their derivatives. As we identify \(M\) with the subspace of constant loops in \(\lop M\), these densities correspond to holomorphic functions on \(M\). To simplify notation and statements, we frequently omit explicitly writing the spatial variable \(x\) and treat these densities in the context as functions on \(M\).

The definition of the local Poisson bracket implies that the matrix \(F_2^{ij}(u)\) defines a finite-dimensional Poisson structure on \(M\), and the matrix \(\Omega_2^{ij}(u)\) is symmetric. Moreover, when the local Poisson bracket is expressed in other coordinates on \(M\), the matrices \(F_2^{ij}(u)\), \(\Omega_2^{ij}(u)\), and \(S_{2,k}^{ij}(u)\) transform as tensors of type \((2,0)\) (\textit{ibid.}). For the transformation properties of \(\Gamma_{2,k}^{ij}(u)\), see Corollary \ref{change of coord Walg}.

\bd \label{dispersionlessDef}
A local Poisson bracket $\{ \cdot, \cdot \}_2$  admits a dispersionless limit if $\{ \cdot, \cdot \}_2^{[-1]} = 0$ and $\{ \cdot, \cdot \}_2^{[0]} \neq 0$. In this case, $\{ \cdot, \cdot \}_2^{[0]}$ defines a Poisson bracket of hydrodynamic type on $\lop M$ and it is said to be nondegenerate  if $\det \Omega_2^{ij} \neq 0$ at some points in $M$.
\ed

The following theorem by Dubrovin and Novikov establishes a connection between contravariant metrics and local Poisson brackets.

\bt\cite{DN}\label{DN thm}
Under the notations in \eqref{loc poiss}, if $\{ \cdot, \cdot \}_2^{[0]}$ is a nondegenerate Poisson bracket of hydrodynamic type, then the matrix $\Omega_2^{ij}(u)$ defines a contravariant flat metric on an open subset of $M$, and $\Gamma_{2,k}^{ij}(u)$ are its contravariant Christoffel symbols.
\et

Notice that the matrix $\O^{ij}_2(u)$ defines a contravariant metric on the open subset 
\[
M_0=\{u\in M : \det(\O^{ij}_2(u))\neq 0\}\subseteq M.
\] 
By a slight abuse of terminology, we say that the metric is defined on $M$. The Christoffel symbols $\Gamma^{ij}_{2,k}(u)$   are determined uniquely on $M_0$ from the system of linear equations
\begin{align*} 
\O^{is}_2\Gamma^{jk}_{2,s} &= \O^{js}_2\Gamma^{ik}_{2,s},\\\nonumber
\Gamma^{ij}_{2,k} + \Gamma^{ji}_{2,k} &= \partial_{u^k} \O^{ij}_2.
\end{align*}
Flatness of the metric means that the corresponding Riemann curvature tensor 
\[
R^{ijk}_l := \O^{is}_2\left(\partial_{u^s}\Gamma^{jk}_{2,l} - \partial_{u^l}\Gamma^{jk}_{2,s}\right) + \Gamma^{ij}_{2,s}\Gamma^{sk}_{2,l} - \Gamma^{ik}_{2,s}\Gamma^{sj}_{2,l}
\]
vanishes identically. Moreover, a function $t(u)$ is called a flat coordinate  if  \be \label{flatcond}
\O^{is}_2\partial_{u^s} \xi_j + \Gamma^{is}_{2,j} \xi_s = 0, \quad i,j=1,\dots,r,~~~\xi_j=\partial_{u^j} t.
\ee
 Thus, the metric $\O^{ij}_2(u)$ is flat if  and only if there  exist locally $r$ functionally independent flat coordinates and in those coordinates the matrix $\O_2^{ij}$ is constant. 
 
Let $\{ \cdot, \cdot \}_1$ and $\{ \cdot, \cdot \}_2$ be compatible local Poisson brackets on $M$ admitting a dispersionless limits whose leading terms define  nondegenerate Poisson brackets of hydrodynamic type
\begin{eqnarray} \label{loc poiss pre}
\{ u^i(x), u^j(y) \}_\alpha^{[0]} = \Omega_\alpha^{ij}(u(x)) \delta'(x-y) + \Gamma_{\alpha,k}^{ij}(u(x)) u_x^k \delta(x-y), \quad \alpha = 1, 2.
\end{eqnarray}
Then the compatibility implies that the pair of matrices $(\Omega_2^{ij}(u), \Omega_1^{ij}(u))$ defines a flat pencil of metrics on $M$. Specifically, $\Omega_2^{ij} + \lambda \Omega_1^{ij}$ defines a flat metric  for any constant $\lambda$ and its Christoffel symbols equal $\Gamma_{2,k}^{ij} + \lambda \Gamma_{1,k}^{ij}$.  See \cite{DFP} for more details on the notion of contravariant metric and  flat pencil of metrics.

\section{Dubrovin-Frobenius manifolds} \label{DFMS}

A Dubrovin-Frobenius manifold \cite{DuRev} is a manifold equipped with a holomorphic  structure of a Frobenius algebra on the tangent space at each point that satisfies certain compatibility conditions. A Frobenius algebra is a commutative, associative algebra with identity $e$ and a nondegenerate bilinear form $\Pi$ that is invariant under the product, i.e., $\Pi(a \cdot b, c) = \Pi(a, b \cdot c)$. The bilinear form $\Pi$ defines a flat metric on the manifold, and the identity vector field $e$ must be constant with respect to it.

Let $M$ be a Dubrovin-Frobenius manifold. Let $(t^1, \ldots, t^r)$ be flat coordinates for $\Pi$ such that $e = \partial_{t^{r-1}}$. Then the compatibility conditions ensure the existence of a function $\mathbb{F}(t^1, \ldots, t^r)$, which encodes the structure of the Dubrovin-Frobenius manifold. This defines the flat metric $\Pi$ in terms of the third derivatives of the potential.
\begin{equation} \label{flat metric}
\Pi_{ij}(t) = \Pi(\partial_{t^i}, \partial_{t^j}) = \partial_{t^{r-1}} \partial_{t^i} \partial_{t^j} \mathbb{F}(t),
\end{equation}
and setting $\Omega_1^{ij}$ as the inverse of the matrix $\Pi_{ij}$, the structure constants of the Frobenius algebra are
\[
C_{ij}^k = \Omega_1^{kp} \partial_{t^p} \partial_{t^i} \partial_{t^j} \mathbb{F}(t).
\]

The associativity of the Frobenius algebra implies that $\mathbb{F}(t)$ satisfies the Witten-Dijkgraaf-Verlinde-Verlinde (WDVV) equations introduced in  \cite{wdvv}:
\begin{equation} \label{frob}
\partial_{t^i} \partial_{t^j} \partial_{t^k} \mathbb{F}(t) ~ \Omega_1^{kp} ~ \partial_{t^p} \partial_{t^q} \partial_{t^n} \mathbb{F}(t) =
\partial_{t^n} \partial_{t^j} \partial_{t^k} \mathbb{F}(t) ~ \Omega_1^{kp} ~ \partial_{t^p} \partial_{t^q} \partial_{t^i} \mathbb{F}(t),
\quad \forall i, j, q, n.
\end{equation}

The definition of a Dubrovin-Frobenius manifold includes the existence of an Euler vector field $E$ satisfying
\[
\Lie_E \mathbb{F}(t) = \left( 3 - d \right) \mathbb{F}(t) + \frac{1}{2} A_{ij} t^i t^j + B_i t^i + c, \quad d, A_{ij}, B_i, c \in \mathbb{C}.
\]
In this article, we assume $E$ takes the form
\[
E = \sum_i d_i t^i \partial_{t^i}, \quad d_{r-1} = 1, \quad d_i \in \mathbb{C}.
\]
A Dubrovin-Frobenius manifold is referred to as polynomial, algebraic, logarithmic, etc., depending on the properties of the corresponding potential $\mathbb{F}(t)$.

For any Dubrovin-Frobenius manifold, there is an associated  flat pencil of metrics. This pencil consists of the metric, called the intersection form, defined by the matrix 
\begin{equation} \label{finding intersection}
\Omega_2^{ij}(t) := \Lie_E (\Omega_1^{ik} \Omega_1^{jm} \partial_{t^m} \partial_{t^k} \mathbb{F}(t)),
\end{equation}
and the flat metric $\Omega_1^{ij}$. Conversely, under certain  conditions, a flat pencil of metrics on $M$ defines a unique (up to equivalence) Dubrovin-Frobenius manifold \cite{DFP}.

\section{Regular nilpotent element in $\g$}

We consider  the general complex Lie algebra $\g$ of rank $r$ with the nondegenerate invariant bilinear form  $\bil {\cdot}{\cdot}$. We denote the Lie bracket by $[\cdot, \cdot]$. Define the adjoint representation $\ad: \g \to \textrm{End}(\g)$ by $\ad_{g_1}(g_2) := [g_1, g_2]$. For $g \in \g$, let $\mathcal{O}_g$ denote the orbit of $g$ under  the adjoint action of the Lie group corresponding to $\g$, and let $\g^g$ denote the centralizer of $g$ in $\g$, i.e., $\g^g := \ker \ad_g$. An element $g$ is called nilpotent if $\ad_g$ is nilpotent in $\textrm{End}(\g)$, and it is called regular if $\dim \g^g = r$.

Let $\e_{i,j}$ denote the $r \times r$ matrix whose $(\nu, \mu)$-entry equals $\delta_{\nu,i}\delta_{\mu,j}$, $i,j=1,\ldots,r$. These matrices form a basis of $\g = \mathfrak{gl}_r$ and satisfy
\[
[\e_{i,j}, \e_{\mu,\nu}] = \delta_{\mu,j} \e_{i,\nu} - \delta_{i,\nu} \e_{\mu,j}, \quad \bil{\e_{i,j}}{\e_{\mu,\nu}} = \delta_{i,\nu} \delta_{\mu,j}.
\]
Following \cite{DS}, we fix the $\sll_2$-triple $\{L_2, h, f\}$:
\[
L_2 = \sum_{i=1}^{r-1} \e_{i,i+1}, \quad h = \frac{1}{2} \sum_{i=1}^{r} (r - 2i + 1) \e_{i,i}, f = \sum_{i=1}^{r-1} i(r-i) \e_{i+1,i}.
\]
Then 
\[
 \quad [L_2, f] = 2h, \quad [h, L_2] = L_2, \quad [h, f] = -f.
\]
 Note that $L_2,~f$ are regular nilpotent elements and $h$ is a regular semisimple element.  In addition, we define
\[
L_i = \sum_{k=1}^{r-i+1} \e_{k,i-1+k}, \quad K_i = \sum_{j=1}^{i-1} \e_{r-i+j+1,j}, \quad i = 1, \ldots, r.
\]
Thus, $L_1$ is the center of $\g$. Under the Dynkin grading
\[
\g = \bigoplus_{i \in \mathbb{Z}} \gr_i, \quad \gr_i := \{g \in \g : \ad_h g = i g\},
\]
we have 
\[
L_i \in \gr_i, \quad K_i \in \gr_{r-i+1}, \quad i = 1, \ldots, r.
\]
We consider the decomposition of $\g$ into $r$ irreducible submodules under the adjoint action of $\{L_2, h, f\}$,
\[
\g = \bigoplus_{i=1}^{r} \V_i, \quad \dim \V_i = 2(i-1) + 1.
\]
The elements $L_1, \dots, L_r$ lie in the centralizer $\g^{L_2}$, form a basis for it, and serve as highest weight vectors of the irreducible $\sll_2$-modules $\V_i$. By construction, $L_i \in \V_i$.

We use the duality between $\g^{L_2}$ and $\g^f$ under the bilinear form $\bil{\cdot}{\cdot}$ (see \cite{wang}) to fix  a basis $\gamma_i$ for $\g^f$ such that
\[
\bil{\gamma_i}{L_j} = \delta_{ij}, \quad i = 1, \ldots, r.
\]
Then $\gamma_i \in \gr_{-j}$ when $L_i \in \gr_j$. 

Define the Slodowy slice $Q := L_2 + \g^f$, and fix coordinates $(u^1, \ldots, u^r)$ on $Q$ such that
\[
Q = L_2 + \sum_{i=1}^r u^i \gamma_i.
\]
From the representation theory of $\sll_2$ subalgebras, we have $\g^f \oplus \ad_{L_2} \g = \g$. Hence, $Q$ is a transverse subspace to the orbit space $\mathcal{O}_{L_2}$ at $L_2$.

We establish the following basis for $\bigoplus_{i \leq 0} \gr_i$
\[
\gamma_i, \ad_{L_2} \gamma_i, \ldots, \frac{1}{(i-1)!} \ad_{L_2}^{i-1} \gamma_i, \quad i = 1, \ldots, r,
\]
and a basis for $\bigoplus_{i \geq 0} \gr_i$
\[
L_i, \ad_{f} L_i, \ldots, \ad_{f}^{i-1} L_i, \quad i = 1, \ldots, r.
\]
Note that, by definition, 
\[
\gamma_1, \gamma_{r-1}, \gamma_{r} \text{ equal } \frac{1}{r} L_1, \frac{1}{2} K_3, K_2, \text{ respectively}.
\]

\bl We have the following orthogonality relation
\[
\bil{\frac{1}{I!} \ad_{L_2}^I \gamma_i}{\ad_f^J L_j} = (-1)^I \frac{(2i-2)!}{(2i-I-2)!} \delta_{ij} \delta^{IJ}.
\]
\el

\begin{proof}
Using the representation theory of $\sll_2$-triples \cite{Humph}, one can prove inductively that
\[
[L_2, \ad_f^I L_j] = I(2j-I-1) \ad_f^{I-1} L_j.
\]
For $I = J = 1$, we compute
\[
\bil{\ad_{L_2} \gamma_i}{\ad_f L_j} = -\bil{\gamma_i}{\ad_{L_2} \ad_f L_j} = \bil{\gamma_i}{[L_j, 2h]} = -2(i-1) \delta_{ij}.
\]
By induction for $I > 1$, we obtain
\begin{eqnarray*}
\bil{\frac{1}{I!} \ad_{L_2}^I \gamma_i}{\ad_f^I L_j} &=& \bil{\frac{1}{I!} \ad_{L_2}^{I-1} \gamma_i}{\ad_{L_2} \ad_f^{I} L_j} \\
&=& -\frac{I(2j-I-1)}{I} \bil{\frac{1}{(I-1)!} \ad_{L_2}^{I-1} \gamma_i}{\ad_f^{I-1} L_j} \\
&=& (-1)^I \frac{(2i-2)!}{(2i-I-2)!} \delta_{ij} \delta^{IJ}.
\end{eqnarray*}
For $I > J$, we recursively reduce $\bil{\ad_{L_2}^I \gamma_i}{\ad_f^J L_j}$ to a zero term proportional to $\bil{\ad_{L_2}^{I-J-1} \gamma_i}{\ad_f L_j}$.
\end{proof}

\section{Drinfeld-Sokolov reduction} \label{revDS}

We review the construction of the classical $W$-algebra associated with the regular nilpotent element $L_2$, following the seminal work of Drinfeld and Sokolov \cite{DS}.

We consider the loop space $\lop \g$ of $\g$ and extend the bilinear form $\bil{\cdot}{\cdot}$ on $\g$ to $\lop \g$ by defining
\beq
(g_1|g_2) = \int_{S^1} \bil{g_1(x)}{g_2(x)} \, dx, \quad g_1, g_2 \in \lop \g.
\eeq
Given a functional $\mathcal{F}$ on $\lop \g$, its variational derivative (or gradient) $\delta \mathcal{F}(g)$ is defined by
\beq
\frac{d}{d\theta} \mathcal{F}(g + \theta \mathrm{w}) \big|_{\theta=0} = \big( \delta \mathcal{F}(g) | \mathrm{w} \big), \quad \textrm{for all } \mathrm{w} \in \lop \g.
\eeq
This leads to the Lie–Poisson bracket $\mathbb{B}$ on $\lop \g$, given for functionals $\mathcal{F}, \mathcal{I}$ by
\begin{eqnarray}\label{Poissbrac}
\{\mathcal{F}, \mathcal{I}\}(g(x)) &:=& \Big( \partial_x \delta \mathcal{I}(g(x)) + [g(x), \delta \mathcal{I}(g(x))] \Big| \delta \mathcal{F}(g(x)) \Big),
\end{eqnarray}
To express the brackets of  $\mathbb{B}$, we fix a basis $\xi_1, \xi_2, \ldots$ for $\g$ and a dual basis $\xi^1, \xi^2, \ldots$, satisfying $\bil{\xi_i}{\xi^j} = \delta_i^j$. We consider the structure constants of $\g$ and the Gram matrix  given by
\beq
[\xi^i, \xi^j] := c_k^{ij} \xi^k, \quad G^{ij} := \bil{\xi^i}{\xi^j}.
\eeq
Then, under the coordinates $q^i:\g\to \mathbb C$ defined by $q^i(g)= \bil{(g - L_2)}{\xi^i}$ for $g\in \g$, the brackets of $\mathbb{B}$  have the form
\beq
\{q^i(x), q^j(y)\} = G^{ij} \delta'(x-y) - c_k^{ij} q^k(x) \delta(x-y).
\eeq

Let us consider the affine subspace $\overline{\mathcal{Q}} \subset \lop \g$ given by
\[
\overline{\mathcal{Q}} := L_2 - \sum_{i=1}^r v^i(x) \e_{r,i} \subset \lop \g.
\]
Then the subspace of constant loops in $\overline{\Q}$ is transverse to the adjoint orbit $\mathcal{O}_{L_2}$ at $L_2$. Drinfeld and Sokolov proved that $\mathbb{B}$ reduces to $\overline{\mathcal{Q}}$ under the gauge action \eqref{gauge action}, and the reduced local Poisson bracket equals  the second Adler-Gelfand-Dickey bracket \eqref{AGD_bracket} (see Theorem 3.22, \cite{DS}). Moreover, they show that different transversal subspaces to the adjoint orbit yield isomorphic local Poisson brackets. 

In this article, instead of $\overline{\mathcal{Q}}$, we consider the affine loop space
\be 
\mathcal{Q} := L_2 + \lop{\g^f}= L_2 + \sum_{i=1}^r u^i(x) \gamma_i,
\ee
associated with the Slodowy slice $Q$ and we denote by $\mathbb B_2^\Q$ the reduction of $\mathbb{B}$ to $\Q$. 

We now give a more detailed construction of $\mathbb{B}_2^{\mathcal{Q}}$, which is central to our analysis. Let $\mathcal{B}$ denote the space of operators of the form
\begin{equation*}\label{op:DS}
\mathcal{L} = \partial_x + b + L_2, \qquad b \in \lop \bneg, \quad \bneg := \bigoplus_{i \leq 0} \gr_i.
\end{equation*}
This space is invariant under the following gauge action:
\be\label{gauge action}
(\mathrm{w}, \mathcal{L}) \mapsto (\exp{\ad{\mathrm{w}}}) \, \mathcal{L}, \quad \mathrm{w} \in \lop \nneg, \quad \mathcal{L} \in \mathcal{B}, \quad \nneg := \bigoplus_{i \leq -1} \gr_i.
\ee
Then for any $\mathcal{L} \in \mathcal{B}$, there exists a unique element $\mathrm{w} \in \lop \nneg$ such that
\begin{equation*}\label{op:fixing}
\mathcal L^{can}=\partial_x + q + L_2 = (\exp{\ad{\mathrm{w}}}) \mathcal{L}, \quad q+L_2 \in \Q.
\end{equation*}
By expansion, we get the following relation between $q$, $\mathrm w$ and $b$
\be \label{op:fixing1}
q - [\mathrm{w}, L_2] = b - \mathrm{w}_x + [\mathrm{w}, b] + \sum_{i > 0} \frac{1}{(i+1)!} \ad_{\mathrm{w}}^i \big(-\mathrm{w}_x + [\mathrm{w}, b] + [\mathrm{w}, L_2]\big).
\ee
Let us  write 
\begin{equation}
b = \sum_{i=1}^{r} \sum_{I=0}^{i-1} b_I^i(x) \frac{1}{I!} \ad_{L_2}^I \gamma_i, \quad \mathrm{w} = \sum_{i=1}^{r} \sum_{I=0}^{i-2} \mathrm{w}_I^i(x) \frac{1}{I!} \ad_{L_2}^I \gamma_i,
\end{equation}
Then using the Dynkin grading and the relation $\g^f \oplus [\nneg, L_2] = \bneg$, provided by the representation theory of $\sll_2$-triples, we derive recursive equations expressing $u^k(x)$ and $\mathrm{w}_I^i(x)$ as differential polynomials in  $b_J^j(x)$. 

Note that  $u^i(x)$ as  differential polynomials, yield a complete set of generators for the ring $R$ of differential polynomials in $b^j_J(x)$ invariant under the gauge action \eqref{gauge action}, i.e., if $P\in R$, then $P$ can be written as differential polynomial in $u^i(x)$.  By assigning the degree of $\partial_x^k b_J^i$ as $k + i - J$, we find that the generators $u^i(x)$ are  quasihomogeneous  polynomials of degree $i$. For example, let $\phi_i : \g \to \gr_i$ denote the projection map. Then the identity \eqref{op:fixing1} leads to 
\begin{eqnarray}\label{dsterms}
u^1(x) &=& b_0^1(x), \\\nonumber
-[\phi_{-1}(\mathrm{w}), L_2] &=& \phi_0(b) - b_0^1(x)\gamma_1, \\\nonumber
u^2(x) \gamma_2 - [\phi_{-2}(\mathrm{w}), L_2] &=& \phi_{-1}\big(b - \mathrm{w}_x + [\mathrm{w}, b] + \frac{1}{2}[\mathrm{w}, [\mathrm{w}, L_2]]\big).
\end{eqnarray}
The set of functionals $\mathcal{R}$ on $\mathcal{Q}$ is defined as the functionals on $\mathcal{B}$ whose densities belong to the ring $R$. It follows that $\mathcal{R}$ is closed under the Poisson bracket $\mathbb{B}$, resulting in the reduced Poisson bracket $\mathbb{B}_2^{\mathcal{Q}}$.
 Furthermore, the brackets $\{\cdot, \cdot\}_2^\Q$ of $\mathbb{B}_2^\mathcal{Q}$ can be computed using  Leibniz rule:
\begin{equation}\label{leibniz rule}
\{u^\mu(x), u^\nu(y)\}_2^\mathcal{Q} := \frac{\partial u^\mu(x)}{\partial (b_I^i)^{(k)}} \partial_x^k \Bigg( \frac{\partial u^\nu(y)}{\partial (b_J^j)^{(l)}} \partial_y^n \big(\{b_I^i(x), b_J^j(y)\}\big) \Bigg),
\end{equation}
where the entries on the right-hand side are expressed entirely in terms of the densities $u^i(x)$ and their derivatives. By definition, we have 
\begin{eqnarray}\label{lie-poiss-in-b}
\{b_I^i(x), b_J^j(y)\} &=& \frac{1}{\Theta_I^i} \frac{1}{\Theta_J^j} \Big( \bil{\ad_f^J L_j}{\ad_f^I L_i} \partial_x + \bil{b}{[\ad_f^J L_j, \ad_f^I L_i]}\Big) \delta(x-y), \\\nonumber
\Theta_I^i &:=& \bil{\frac{1}{I!} \ad_{L_2}^I \gamma_i}{\ad_f^I L_i} = (-1)^I \frac{(2i-2)!}{(2i-I-2)!}.
\end{eqnarray}

 \bp \label{linear part}
The linear terms of $u^i(x)$ are given by
\be \label{linear term}
\sum_{I=0}^{i-1} \frac{(-1)^I}{I!} \partial_x^I b_I^i.
\ee
In particular, $u^i(x)$ depends linearly on $b_0^i(x)$.
\ep

\begin{proof}
We introduce a spectral parameter $\epsilon$ and set $\mathcal{L}(\epsilon) = \partial_x + \epsilon b + L_2$. Let $\mathrm{w}(\epsilon)$ and $\mathcal{L}^{\text{can}}(\epsilon)$ denote the corresponding operators. At $\epsilon = 0$, we have $\mathcal{L}(0) = \partial_x + L_2$, $\mathrm{w}(0) = 0$, and $\mathcal{L}^{\text{can}}(0) = \mathcal{L}(0)$. Differentiating the relation
\be\label{terms e}
\mathcal{L}^{\text{can}}(\epsilon) = \mathcal{L}(\epsilon) + [\mathrm{w}(\epsilon), \mathcal{L}(\epsilon)] + \frac{1}{2}[\mathrm{w}(\epsilon), [\mathrm{w}(\epsilon), \mathcal{L}(\epsilon)]] + \ldots
\ee
with respect to $\epsilon$ and evaluating at $\epsilon = 0$, we obtain
\begin{eqnarray*}
q'(0) &=& b + [\mathrm{w}(0), \mathcal{L}'(0)] + [\mathrm{w}'(0), \partial_x + L_2] \\
&=& b + [\mathrm{w}'(0), \partial_x + L_2] \\
&=& b - \mathrm{w}_x'(0) + [\mathrm{w}'(0), L_2].
\end{eqnarray*}
Since $[\mathrm{w}'(0), L_2]$ does not contribute to $q'(0)$, the coordinate of $\gamma_i$ satisfies
\beq
(u^i(x))'(0) = b_0^i(x) - (\mathrm{w}_x'(0))_0^i,
\eeq
where we write $\mathrm{w}'(0) = \sum_{i=1}^n \sum_{I>0} (\mathrm{w}'(0))_I^i \frac{1}{I!} \ad_{L_2}^I \gamma_i$. For $I > 0$, the coefficients of $\frac{1}{I!} \ad_{L_2}^I \gamma_i$ yield the recursive relations
\beq
(\mathrm{w}'(0))_{I-1}^i = \frac{1}{I+1}\big(-(\mathrm{w}_x'(0))_I^i + b_I^i(x)\big).
\eeq
These equations lead to the expression in \eqref{linear term}.
\end{proof}

\bc 
The quadratic terms $\ddot u^i(x)$ and $\ddot {\mathrm{w}}_I^i(x)$ of $u^i(x)$ and $\mathrm{w}_I^i(x)$ are recursively determined by
\be \label{quadratic terms}
\ddot q - [\ddot {\mathrm{w}},L_2] = -\mathrm{w}_x + [\mathrm{w}'(0), b] +\frac{1}{2} [\mathrm{w}'(0), [\mathrm{w}'(0), \partial_x + L_2]].
\ee
\ec
\begin{proof} The quadratic terms are given by applying the operator $\dfrac{1}{2}\frac{d^2}{d\epsilon^2}|_{\epsilon=0} $ to equation 
 \eqref{terms e}.
\end{proof}
\subsection{Classical $W$-algebra}\label{classwalg}

In this section, we utilize known results on the Drinfeld–Sokolov reduction associated with nilpotent elements in simple Lie algebras to show that $\mathbb{B}_2^Q$ defines a classical $W$-algebra.  

We consider  the $\sll_2$-triple $\{L_2, h, f\}$  in the special linear Lie algebra $\sll_r$ and the associated affine loop space  \[\widetilde{\mathcal{Q}} := L_2 + \lop{\sll_r^f}=L_2+\sum_{i=2}^r u^i(x)\gamma_i.\] 
Note that the index $i$  runs from 2 to $r$. Then, we perform Drinfeld-Sokolov gauge action in the same manner as  for $\Q$, i.e., by restricting  the procedure and the local Poisson bracket \eqref{Poissbrac} to ${\sll_r}$ (see \cite{DS} for details). This yields a local Poisson bracket $\mathbb B^{\widetilde{\mathcal{Q}}}_2$ on $\widetilde{\mathcal{Q}}$. 

Writing the brackets  $\{.,.\}^{\widetilde{\mathcal{Q}}}_2$  in the form \ref{loc poiss}, the finite dimensional Poisson bracket defined by the matrix $F^{ij}(u),~i,j=2,\ldots,r$ on $L_2+\sll_r^f$ coincides  with the transverse Poisson structure of the Lie-Poisson structure on $\sll_r$. It is known that the symplectic leaves  of the Lie-Poisson structure  coincide with the adjoint orbits.   Since the nilpotent element $L_2$ is regular, its adjoint orbit  is of maximum dimension. Hence, the transverse Poisson bracket is trivial and $\{.,.\}^{\widetilde{\mathcal{Q}}}_2$  admits a dispersionless limit, i.e., $F^{ij}(u)=0$, for $i,j=2,\ldots,r$ (see Proposition 4.4 of \cite{mypaper1}). 

According to the work in \cite{BalFeh1},   $\{.,.\}^{\widetilde{\mathcal{Q}}}_2$  is a classical $W$-algebra. Precisely, 
\begin{eqnarray*}
\{u^2(x), u^2(y)\}_2^{\widetilde{\mathcal{Q}}} &=& c \delta'''(x - y) + 2u^2(x) \delta'(x-y) + z^2_x \delta(x-y),~~c\neq 0\in \mathbb C \\\nonumber
\{u^2(x), u^j(y)\}_2^{\widetilde{\mathcal{Q}}} &=& j u^j(x) \delta'(x-y) + (j-1) u^j_x \delta(x-y),~~j\neq 2. 
\end{eqnarray*}
This leads to the following Proposition. 

\bp \label{classical 1}
The local Poisson bracket $\mathbb{B}_2^\mathcal{Q}$ admits a dispersionless limit.  Moreover, under the following change of coordinates on $Q$
\be \label{coordz}
z^2 = u^2 + \frac{1}{2} (u^1)^2 \bil{\gamma_1}{\gamma_1}, \quad z^i = u^i, \quad i \neq 2,
\ee
the Poisson brackets satisfy the  identities defining a classical $W$-algebra,
\begin{eqnarray}\label{walgebra2}
\{z^2(x), z^2(y)\}_2^\mathcal{Q} &=& c \delta'''(x-y) + 2z^2(x) \delta'(x-y) + z^2_x \delta(x-y), \\\nonumber
\{z^2(x), z^j(y)\}_2^\mathcal{Q} &=& j z^j(x) \delta'(x-y) + (j-1) z^j_x \delta(x-y),~~j\neq 2,
\end{eqnarray}
for some nonzero constant $c$.
\ep
\begin{proof}
We observe that  
\be \label{reductionto sl}
\g = \sll_r \oplus \mathbb{C} \Y_1, \quad \mathcal{Q} = \widetilde{\mathcal{Q}} + u^1(x) \Y_1, \quad [\Y_1, \g] = 0, \quad \bil{\Y_1}{\sll_r} = 0.\ee 
In particular, from the gauge action, the densities  $u^2(x), \ldots, u^r(x)$  are independent of $b_0^1(x)$ but $u^1(x) = b_0^1(x)$. Hence, we get from equation \eqref{leibniz rule}  
\beq
\{u^1(x), u^1(y)\}_2^\Q = \bil{L_1}{L_1} \delta'(x-y) = r \delta'(x-y), \quad \{u^1(x), u^i(y)\}_2^\Q = 0, \quad i \neq 1.
\eeq
Furthermore, \beq \{u^i(x),u^j(y)\}_2^\Q=\{u^i(x),u^j(y)\}_2^{\widetilde{\Q}},~~i,j\neq 1. \eeq
 Thus, for $z^1(x)$, we have
\[
\{z^2(x), z^1(y)\}_2^\mathcal{Q} = \frac{1}{r} u^1(x) \{u^1(x), u^1(y)\}_2^\mathcal{Q} = u^1(x) \delta'(x-y).
\]
This completes the proof.
\end{proof}

\section{Entries of classical $W$-algebra} \label{entriesvalg}

In this section, we investigate how the entries of the classical $W$-algebra $\mathbb{B}_2^\mathcal{Q}$ depend on $u^{r-1}(x)$. Specifically, we aim to identify monomials in $\{u^i(x), u^j(y)\}_2^\mathcal{Q}$ that are proportional to $(u^{r-1}(x))^n$ for some natural number $n$. Proposition \ref{linear part} shows that such monomials arise from terms proportional to $(b_0^{r-1}(x))^n$ in the expansion \eqref{leibniz rule}.

Proposition \ref{classical 1} demonstrates that the Poisson brackets are linear in $u^{r-1}(x)$ when $r = 2$. However, this linearity does not hold for $r \geq 3$. Hence, for this section until the remainder of Section 8, we assume $r \geq 3$.

We begin by analyzing the quasihomogeneity of $\mathbb{B}_2^\mathcal{Q}$. Assigning degrees as $\deg \partial_x^j u^i = i + j$ and $\deg \partial_x^k b_J^i = k + i - J$, we proceed with the following definitions.

\bd
A matrix $A_{IJ}^{ij}(b)$ with entries that are differential polynomials in the densities $b_S^s(x)$ is said to be homogeneous of degree $n$ if each entry $A_{IJ}^{ij}(b)$ is a quasihomogeneous polynomial of degree $\deg b_I^i + \deg b_J^j + n$. Similarly, a matrix $A^{ij}(u)$ with entries that are differential polynomials in $u^k(x)$ is homogeneous of degree $n$ if each entry $A^{ij}(u)$ is a quasihomogeneous polynomial of degree $\deg u^i + \deg u^j + n$.
\ed

The local Poisson bracket $\mathbb{B}$ on $\lop \g$ is the sum of
\be \label{general Poiss bracket}
\{b_I^i(x), b_J^j(y)\}^{[0]} := A_{IJ}^{ij}(b) \delta'(x-y), \quad \{b_I^i(x), b_J^j(y)\}^{[-1]} := B_{IJ}^{ij}(b) \delta(x-y).
\ee
The term $\{\cdot, \cdot\}^{[-1]}$ corresponds to the Lie-Poisson bracket on $\lop \g$ restricted to $\lop \bneg$. The Dynkin grading of $\g$ implies that the matrix $B_{IJ}^{ij}(b)$ is homogeneous of degree $-1$. On the other hand, $\{\cdot, \cdot\}^{[0]}$ is defined by restricting the invariant bilinear form $\bil{\cdot}{\cdot}$ to $\lop \bneg$. Consequently, the matrix $A_{IJ}^{ij}(b)$ is constant, with nonzero entries occurring only when $i - I = 0 = j - J$. Applying the Leibniz rule yields
{\small
\begin{eqnarray}\label{libniz expand}
\{u^\mu(x), u^\nu(y)\}_2^\mathcal{Q} &=& \frac{\partial u^\mu(x)}{\partial (b_I^i)^{(l)}} \partial_x^l \Bigg( \frac{\partial u^\nu(y)}{\partial (b_J^j)^{(h)}} \partial_y^h \big(A_{IJ}^{ij}(b) \delta'(x-y) + B_{IJ}^{ij}(b) \delta(x-y)\big) \Bigg) \\ \nonumber
&=& \sum_{h \geq \alpha \geq 0} \sum_{l \geq \beta \geq 0} (-1)^h {h \choose \alpha} {l \choose \beta} \frac{\partial u^\mu(x)}{\partial (b_I^i)^{(l)}} \Bigg( B_{IJ}^{ij}(b) \Bigg( \frac{\partial u^\nu(x)}{\partial (b_J^j)^{(h)}} \Bigg)^{(\alpha)} \Bigg)^{(\beta)} \delta^{(h+l-\alpha-\beta)}(x-y) \\ \nonumber
&+& \sum_{h \geq \alpha \geq 0} \sum_{l \geq \beta \geq 0} (-1)^h {h \choose \alpha} {l \choose \beta} \frac{\partial u^\mu(x)}{\partial (b_I^i)^{(l)}} \Bigg( A_{IJ}^{ij}(b) \Bigg( \frac{\partial u^\nu(x)}{\partial (b_J^j)^{(h)}} \Bigg)^{(\alpha)} \Bigg)^{(\beta)} \delta^{(h+l-\alpha-\beta+1)}(x-y).
\end{eqnarray}
}

We write the brackets $\{u^i(x), u^j(y)\}_2^\mathcal{Q}$ of $\mathbb{B}_2^\mathcal{Q}$ in the form \eqref{loc poiss pre} and study the matrix $\Omega_2^{ij}(u)$. From Proposition  \ref{classical 1}, the matrix $F_2^{ij}(u) = 0$.

\bp \label{exponent in u}
The maximum exponent of $u^{r-1}(x)$ in the expansions of $\{u^i(x), u^j(y)\}_2^\mathcal{Q}$ is 2. Furthermore, the term with exponent 2 may appear only in the entry $\Omega_2^{rr}(u)$.
\ep

\begin{proof}
We collect the coefficients of $\delta^{(m)}(x-y)$ in the expansion \eqref{libniz expand} and write
\beq
\{u^i(x), u^j(y)\}_2^\mathcal{Q} = \sum_{m \geq 0} T_m^{ij}(u,u_x,u_{xx},\ldots) \delta^{(m)}(x-y).
\eeq
We claim that, for any $m$, the matrix $T_m^{ij}$ is homogeneous of degree $-(m+1)$. Indeed, the degree of the coefficient of $\delta^{(h+l-\alpha-\beta)}(x-y)$ in \eqref{libniz expand} is, omitting writing the independent variable $x$,
\begin{eqnarray*}
\deg u^\mu - \deg b_I^i - l + \deg u^\nu - \deg b_J^j - h &+ & \alpha + \deg b_I^i + \deg b_J^j - 1 + \beta\\ &=& \deg u^\mu + \deg u^\nu - (l + h - \alpha - \beta + 1).
\end{eqnarray*}
Similarly, using the properties of the matrix $A_{IJ}^{ij}(b)$, the degree of the coefficient of $\delta^{(h+l-\alpha-\beta+1)}(x-y)$ is $\deg u^\mu + \deg u^\nu - (l + h - \alpha - \beta + 2)$. As
\[
\deg T_m^{ij} = i + j - m - 1 \leq 2r - m - 1.
\]
and $\deg u^{r-1} = r-1$, the maximum exponent of $u^{r-1}(x)$ in $T_m^{ij}$ is 2. In particular, $(u^{r-1}(x))^2$ can appear only if $m = 1$ or $m = 0$. Thus, considering the forms $\{\cdot,\cdot\}_2^{[k]}$ of the expansion \eqref{loc poiss pre}, it follows that $\Omega_2^{rr}(u)$ is the only entry of the local Poisson bracket $\mathbb{B}_2^\mathcal{Q}$ that may contain a monomial proportional to $(u^{r-1}(x))^2$.
\end{proof}

\bl \label{metric:z1}
The matrix $\Omega^{ij}_2(u)$ is homogeneous of degree $-2$. Consequently, it is a lower antidiagonal matrix with respect to $u^{r-1}$, i.e., $\partial_{u^{r-1}} \Omega^{ij}_2(u) = 0$ for $i + j < r+1$. Moreover, $\Omega^{1i}_2(u) = r \delta^{1i}$, and the coefficient of $(u^{r-1})^2$ in $\Omega^{rr}_2(u)$ equals $\frac{r-1}{r}$.
\el

\begin{proof}
We need only to determine the coefficient of $(u^{r-1}(x))^2$, other statements follow directly from Propositions \ref{exponent in u} and \ref{classical 1}. We begin by  analyzing the occurrence of $(b_0^{r-1}(x))^2$ in the expansion of $\{u^r(x), u^r(y)\}$. Observe that the terms in $u^r(x)$ involving $b_0^{r-1}(x)$ take the form $b_0^{r-1}(x) S(b)$, where $S(b)$ is a polynomial of degree 1. They  are part of the quadratic terms of $u^r(x)$ which are given by equation \eqref{quadratic terms}. Form this equation, we note that $\ddot{\mathrm{w}}$ and $\mathrm w'(0)$ belong to the image of $\ad_{L_2}$. Thus, they do not depend on $b_0^{r-1}(x)$. Hence, to find $S(b)$, we  consider only the term $[{\mathrm{w}}'(0), b]$ of equation \eqref{quadratic terms}. Then applying $\bil{L_r}{\cdot}$ to both sides,  leads to 
\be \label{coefficient of u1}
S(b) = \bil{[\gamma_{r-1}, L_r]}{\mathrm{w}'(0)}.
\ee
Furthermore, from Leibniz rule,  $b_0^{r-1}(x) \{S(b(x)), u^r(y) - b_0^{r-1}(y) S(b(y))\}$ can also yield a quadratic term in $b_0^{r-1}(x)$. This arises from the Lie-Poisson bracket $\{\cdot, \cdot\}^{[-1]}$ on $\lop \g$. Analyzing this using grading and quasihomogeneity, we may find an additional quadratic term from $b_0^{r-1}(x) \{S(b), \partial_y b_1^r(y)\}^{[-1]}$. In conclusion, the coefficient of $(u^{r-1}(x))^2$ in $\{u^r(x), u^r(y)\}_2^\Q$ equals the coefficient of $(b_0^{r-1}(x))^2$ appearing in the expansion of
\be \label{coefficient of u}
\{\widetilde{u}^r(x), \widetilde{u}^r(y)\}, \quad \widetilde{u}^r(x) := b_0^{r-1}(x) S(b) - \partial_y b_1^r.
\ee
To find the value of $S(b)$, we note that  
\[
[\gamma_{r-1}, L_r] = \frac{1}{2}[\e_{r-1,1} + \e_{r,2}, \e_{1,r}] = \frac{1}{2}(\e_{r-1,r} - \e_{1,2}).
\]
Thus, $S(b)$ depends only on the restriction $\widetilde{\mathrm{w}}$ of $\mathrm{w}'(0)$ to the vector space spanned by $\e_{2,1} \in \gr_{-1}$ and $\e_{r,r-1} \in \gr_{-1}$. From equations  \eqref{dsterms}, we only need the restriction $\widetilde{b}$ of $b$ to $\lop{\h}$. Thus, introducing 
\[
\widetilde{\mathrm{w}} = \mathrm{w}^1(x) \e_{2,1} + \mathrm{w}^2(x) \e_{r,r-1}, \quad \widetilde{b} = a^1(x) L_1 + \sum_{i=1}^{r-1} a^{i+1}(x) (\e_{i,i} - \e_{i+1,i+1}),
\]
and using
\[
[\widetilde{\mathrm{w}}, L_2] = \mathrm{w}^1(x) (\e_{2,2} - \e_{1,1}) + \mathrm{w}^2(x) (\e_{r,r} - \e_{r-1,r-1}).
\]
It  follows that
\be
\mathrm{w}^1(x)= a^2(x), \quad \mathrm{w}^2(x) = a^r(x), \quad S(b) = \frac{1}{2}(-a^2(x) + a^r(x)).
\ee
We now compute the Poisson brackets for the coordinates involved in the expression of $\widetilde{u}^r$. Recall that $b_1^r$ is the coefficient of $\ad_{L_2} K_r = \e_{r-1,1} - \e_{r,2}$. Let $e_1^*, \ldots, e_{r+1}^*$ denote the dual basis of $\gamma_1, \e_{1,1} - \e_{2,2}, \ldots, \e_{r-1,r-1} - \e_{r,r}, \e_{r-1,1} - \e_{r,2}$ under $\bil{\cdot}{\cdot}$. Note that   $e_2^*,\ldots, e_{r}^*$ can be considered as the fundamental weights of the Lie algebra $\sll_{r}$. Specifically, from \cite{Humph}, we get 
\[
e_1^* = L_1, \quad e_{r+1}^* = \frac{1}{2} (\e_{1,r-1} - \e_{2,r}), \quad e_j^* = \sum_{i=1}^{j-1} \e_{i,i} - \frac{j-1}{r} L_1, \quad j = 2, \ldots, r.
\]
Direct computation yields
\beq
\bil{e_2^*}{e_2^*} = \bil{e_r^*}{e_r^*} = \frac{r-1}{r}, \quad \bil{e_2^*}{e_r^*} = \frac{1}{r}, \quad [e_2^*, e_{r+1}^*] = \frac{1}{2} \e_{1,r-1}, \quad [e_r^*, e_{r+1}^*] = -\frac{1}{2} \e_{2,r}.
\eeq
This leads to Poisson brackets
\begin{eqnarray*}
\{a^2(x), a^r(y)\} &=& \frac{1}{r} \delta'(x-y), \\
\{a^2(x), a^2(y)\} &=& \{a^r(x), a^r(y)\}_2 = \frac{r-1}{r} \delta'(x-y), \\
\{-a^2(x) + a^r(x), b_1^r(y)\} &=& -\frac{1}{2} b_0^{r-1}(x) \delta(x-y).
\end{eqnarray*}
Using the expansion \eqref{libniz expand}, we arrive to 
\begin{eqnarray*}
\{\widetilde{u}^r(x), \widetilde{u}^r(y)\} &=& \frac{1}{4} (b_0^{r-1}(x))^2 \big(\bil{e_r^*}{e_r^*} - 2 \bil{e_r^*}{e_2^*} + \bil{e_2^*}{e_2^*}\big) \\
&-& \frac{1}{2} b_0^{r-1}(x) \big(\{-a^2(x) + a^r(x), b_1^r(y)\}\big) + \text{terms free of } (b_0^{r-1}(x))^2.
\end{eqnarray*}
Thus, the coefficient of $(u^{r-1}(x))^2$ is
\be
\frac{r-1}{2r} - \frac{1}{2r} + \frac{1}{2} = \frac{r-1}{r}.
\ee
\end{proof}

Finally, the following corollary shows that the form of some entries of the local Poisson brackets remain invariant under quasihomogeneous changes of coordinates. These entries are observed in the definition of classical $W$-algebra (see \eqref{walgebra2}).

\bc\label{change of coord Walg}
Suppose that the brackets of $\mathbb B_2^Q$ in  some coordinates $(z^1, \ldots, z^r)$ on $Q$ give 
\[
\Omega_2^{2j}(z) = j z^j, \quad \Gamma_{2,k}^{2j}(z) = (j-1) \delta_k^j, ~j=1,\ldots,r.
\]
Then under a change of coordinates on $Q$ of the form
\[
s^2 = z^2, \quad s^i = H^i(z^1, \ldots, z^n), \quad i \neq 2,
\]
where $H^i$ is a quasihomogeneous polynomial of degree $i$ when $\deg z^j=j$, we have 
\[
\Omega_2^{2j}(s) = j s^j, \quad \Gamma_{2,k}^{2j}(s) = (j-1) \delta_k^j, ~j=1,\ldots,r.
\]

\ec

\begin{proof}

Let us  introduce  the Euler vector field
\beq
E':=\sum_i i z^i { \partial_{z^i}}.
\eeq
Then the formula for change of coordinates   gives
\beq
\O^{2j}(s)={\partial_{z^i} s^2 } {\partial_{z^k} s^j}~ \O^{ik}(z)= E'(s^j)=j s^j.
\eeq
 For    $\Gamma^{2j}_{k}(s)$, the change of coordinates leads to 
\begin{eqnarray*}
\Gamma^{2j}_{2,k} d s^k&=& \Big({\partial_{z^i} s^2 } {\partial_{z^c}\partial_{z^l} s^j} \O_2^{i l}(z)+ {\partial_{z^i} s^2} {\partial_{z^l} s^j  } \Gamma^{i l}_{2,c}(z)\Big) d z^c,
\\\nonumber &=& \Big( E' ({\partial_{z^c} s^j  })+  {\partial_{z^l} s^j  } \Gamma^{2l}_{2,c}\Big) d z^c,\\\nonumber
&=& \Big( (j-c){\partial_{z^c} s^j  }+  (c-1){\partial_{z^c} s^j  } \Big) d z^c=(j-1) {\partial_{z^c} s^j } d z^c =(j-1) d s^j.
\end{eqnarray*}
\end{proof}
\subsection{Antidiagonal entries}

In this section, we compute the coefficients of $u^{r-1}(x)$ in the antidiagonal entries $\O^{k,r-k+1}_2(u)$. These coefficients are constants since $\deg \O^{k,r-k+1}_2(u)=k+(r-k+1)-2=r-1$. From  Proposition \ref{linear part}  and  properties of the matrix $A^{ij}_{IJ}(b)$,  we need only to consider the appearance of  $b_0^{r-1}(x)$ in the following coefficient of $\delta'(x-y)$
\beq 
\sum_{i,j,I,J}\sum_{h,l}(-1)^h (l+h) B_{IJ}^{ij}(b)\dfrac{\partial u^k(x) }{ \partial (b_{I}^{i})^{(l)} }\Big(\dfrac{\partial u^{r-k+1}(x) }{ \partial (b_{J}^{j})^{(h)} }\Big)^{h+l-1}.
\eeq 
Let us define the structure constants $\Delta_{j}^{Jt}$ by 
\be\label{constss} [\gamma_{r-1}, \ad_f^J L_j]=\sum_t \Delta_{j}^{Jt} \dfrac{1}{ m!}\ad_{L_2}^m \ga_t;~~ m=t+j-J-r\geq 0.\ee
Here, the integer $m $ is  determined by the Dynkin grading of $\g$. Then applying $\bil {\ad_f^I L_i}{\cdot}$, we get the  coefficient of $b_0^{r-1}(x)$ in the entry $B_{IJ}^{ij}(b)$ equals
\beq \dfrac{1}{\Theta_{J}^{j} } \delta^{Im}\delta^{it} \Delta_{j}^{Jt}\delta(x-y)=  {\Delta_{j}^{Ji} \over \Theta_{J}^{j} }  \delta(x-y),~~~I=i+j-J-r
\eeq
where $\Theta_{J}^{j}= (-1)^J \frac{(2j-2)!}{(2j-J-2)!}$. We conclude that  the coefficient of $u^{r-1}$ in $\O_2^{k,r-k+1}(u)$ is included in the expression
\begin{eqnarray}\label{g1:matrix}
\f^{k}&=&\sum_{i,J}\sum_{h,l}(-1)^h (l+h) {\Delta_{j}^{Ji} \over \Theta_{J}^{j} }{\partial u^k(x) \over \partial (b_{I}^{i})^{(l)} }\Big({\partial u^{r-k+1}(x) \over \partial (b_{J}^{j})^{(h)} }\Big)^{h+l-1},~~I=i+j-J-r.
\end{eqnarray}

\bl\label{metric:z2}
The coefficient of $u^{r-1}$ in  the entry 
 $\O^{k,r-k+1}_2(u)$  equals  $r-1$ for $k=2\ldots,r-1$ and  $\O^{1r}(u)=\O^{r1}(u)=0$.
\el
\begin{proof}
 Note that $\O^{1r}_2(u)=\O^{r,1}_2(u)=0$ follows from the proof of Proposition \ref{classical 1}. Thus, we only consider $\f^k$, $k\neq 1,r$. To get a constant coefficient from $\f^k$, we must have $h+l-1=0$ and $u^k(x)$ is linear in $(b^i_I)^{(l)}$ and $u^{r-k+1}(x)$ is linear in $(b^j_J)^{(h)}$.  Consider the case $h=0$ and $l=1$. It follows from Proposition \ref{linear part}  that $j=r-k+1, J=0, i=k$ and  $I=1$. Thus, we have  
\beq 
{\partial u^{r-k+1}(x) \over \partial (b_{J}^{j})^{(h)} }=1,~~ {\partial u^k(x) \over \partial (b_{I}^{i})^{(l)} }=-1. ~~
\eeq 
and we get from \eqref{constss} the constant 
\beq
-{\Delta_{r-k+1}^{0u}\over \Theta^{r-k+1}_0}=-\Delta_{r-k+1}^{0u}={1\over 2(k-1)}\bil {ad_f L_k}{[\gamma_{r-1},L_{r-k+1}]}.
\eeq 
A  similar analysis for the case  $h=1$ and $l=0$, we get the value  
 \[{1\over 2(r-k)}\bil {ad_f L_{r-k+1}}{[\gamma_{r-1},L_{k}]}.\] 
We note that 
\bea \nonumber
{ad_f L_{r-k+1}}&=&  k(r-k)\Big(\e_{k+1,r}-\e_{1,r-k}\Big)+\sum_{i=1}^{k-1}\Big(i(r-i)-(i+r-k)(k-i)\Big)\e_{i+1,i+r-k},\\ \nonumber
{[\gamma_{r-1},L_{k}]} &=& \frac{1}{2}(\e_{r-1,k}+\e_{r,1+k}-\e_{r-k,1}-\e_{r-k+1,2}).
\eea 
Hence,  
\bea \nonumber 
\bil {ad_f L_{r-k+1}}{[\gamma_{r-1},L_{k}]}&=&2(r-k)+[(k-1)(r-k+1)-(r-1)]-[(r-1)-(r-k+1)(k-1)]\\\nonumber
&=&2(r-k)(k-1). 
\eea 
Therefore,
\beq 
\partial_{u^{r-1}}\Omega^{k,r-k+1}_2={1\over 2(k-1)}\bil {ad_f L_k}{[\gamma_{r-1},L_{r-k+1}]}+{1\over 2({r-k})}\bil {ad_f L_{r-k+1}}{[\gamma_{r-1},L_{k}]}=r-1. 
 \eeq 

\end{proof}

\section{Invariant coordinates}
In this section, we employ the invariant polynomials of $\g$ under the adjoint group action to establish coordinates for Slodowy slice $Q$. Then we examine the change of the entries of $\mathbb{B}^Q_2$ under these coordinates.

 Recall that by  Chevalley's theorem, the ring of  invariant polynomials under the adjoint group action on $\g$ is generated by $r$ homogeneous polynomials with degrees $1,2,\ldots,r$. Moreover, $ P_i=\frac{1}{i}\mathrm{Tr}(g^i),~ g\in \g,~i=1,\ldots,r$  form a complete set of homogeneous  generators with  $\deg P_i=i$. Let $z^i$ be the restriction of $P_i$ to $Q$, i.e.,  
 \[ z^i=\dfrac{1}{i} Tr(g^i),~~g\in Q\]
 Then it follows from Section 2.5 of \cite{sldwy2} that $z^i$ is a quasihomogeneous polynomial of degree $i$ in the coordinate $(u^1,\ldots,u^r)$ with $\deg u^j=j$. 

\bp \label{nice coordinates1} The
  functions $(z^1,\ldots,z^{r})$ define  coordinates on $Q$ and have the form
 \be \label{norm-coord eq}
 z^i=  \left\{
  \begin{array}{ll} u^1, &i=1,\\
  u^2+\dfrac{1}{2r} (u^1)^2,& i=2,\\
  u^{i}+\widetilde z^{i}(u),& i=3,\ldots,r-1,\\
   u^{r}+ \dfrac{(r-1)}{r}u^1 u^{r-1}+\widetilde z^{r}(u), & i=r.
  \end{array}
\right.
\ee
Here,  \be \label{constrain1} \dfrac{\partial \widetilde z^i}{\partial u^{r-1}}=\dfrac{\partial \widetilde z^{i}}{\partial u^{i}}=0, ~i\geq 3.\ee   
 \ep
 \begin{proof}
      The forms of  $z^1$ and $z^2$ are obtained by direct computations and the conditions \eqref{constrain1} follow from quasihomogeneity. Let us assume $i>2$. Let $q(\epsilon)\in Q$ be the element given by replacing $u^i\to\epsilon u^i$. Then the linear term of  $z^i$ is given by \beq \dfrac{d}{d\epsilon}|_{\epsilon=0}   z^i(q(\epsilon))=\Tr\Big(q'\circ L_2^{i-1}\Big)=\bil {q'}{L_{i}}=u^{i},~i=1,\ldots,r.\eeq
 Thus, $(z^1,\ldots,z^r)$ are coordinates on $Q$. The quadratic form of $z^{r}$  is obtained from  the evaluation  
 \beq 
 \dfrac{1}{2}\frac{d^2}{d\epsilon^2}|_{\epsilon=0}  z^{r}=\Tr\Big( (q')^2L_{r-1}+\dfrac{1}{2}\sum_{j=1}^{r-3} 
  q' L_{j+1} q'L_{r-j-1}\Big).
  \eeq
  We write 
  \beq q'=u^r\e_{r,1}+\frac{1}{2}u^{r-1}(\e_{r-1,1}+\e_{r,2})+\frac{1}{r}u^1\sum_{k=1}^r\e_{k,k}.\eeq 
  Then 
  \beq q'L_{j+1}=u^r\e_{r,j+1}+\frac{1}{2}u^{r-1}(\e_{r-1.j+1}+\e_{r,j+2})+\frac{1}{r}u^1\sum_{k=1}^{r-j}\e_{k,j+k}.\eeq
We get $q'L_{r-j-1}$ by replacing $j$ by $r-j-2$ in the formula above. Then the result follows from the values  
\beq \Tr (q' L_{j+1} q'L_{r-j-1})=\Tr\big((q')^2L_{r-1}\big)=\frac{2}{r}u^1 u^{r-1}, ~ j=1,\ldots,r-3.\eeq
\end{proof}

\bl \label{entries in z}
In the coordinates $(z^1,\ldots,z^r)$, the maximum power of $z^{r-1}(x)$ in the corresponding brackets of $\mathbb B_2^\Q$ is 2 and it appears only on the entry $\Omega_2^{rr}(z)$ with coefficient ${(r-1)}$. Moreover, the matrix   $\Omega^{ij}_2(z)$  is a lower antidiagonal for the coordinate $z^{r-1}(x)$, i.e., $\partial_{z^{r-1}}\Omega^{i,j}(z)=0$ for $i+j<r+1$. In addition, \[\Omega^{11}_2(z)=r, ~\O_2^{2i}(z)=iz(x), ~\Gamma^{2i}_{2,k}(z)= (i-1)\delta_{k}^i,~~\partial_{z^{r-1}}\Omega_2^{i,r-i+1}(z)=
  r-1, ~ i=1,\ldots,r.\]
\el
\begin{proof}
 From Proposition \ref{exponent in u}, it follows that the maximum  power of $z^{r-1}(x)$ is $2$ and it may appears only in the matrix entry $\Omega^{rr}_2(z)$.   We treat $\O_2^{ij}(u)$ as $(2,0)$ tensor on $Q$. Then to find the coefficient of $(z^{r-1})^2$  in $\Omega_2^{rr}(z)$, it is enough to compute the values of the one form  \[A:= du^r+\dfrac{r-1}{r}(u^1 du^{r-1}+ u^{r-1}du^1).\] under the tensor $\O_2^{ij}(u)$. This leads to the coefficient of $(u^{r-1})^2$ in the expression 
 \[\O^{rr}_2(u)+\frac{r-1}{r} (u^{r-1})^2\O^{11}_2.\]
 which is $r-1$. Similar computations by evaluating the one form $A$ with $dz^1=du^2$ leads to the coefficient of $z^{r-1}$ in $\O_2^{1r}(z)$.
To find the coefficient of $z^{r-1}$ in   $\Omega_2^{i,r-i+1}(z)$ for $i\neq 1,r$, we similarly use 
 \[ dz^i=du^i+\sum_{k<i} g_k^i(u) du^k, ~\deg g_k^i=i-k. \] 
Then  \[\O_2(dz^i,dz^{r-i+1})=\O^{i,r-i+1}_2(u)+\sum_{k+j<r+1} T_{kj}^i(u)\O^{kj}_2(u),\]
and the terms in the summations do not depends on $u^{r-1}$ by quasihomogeneity and Lemma \ref{metric:z1}.  The remaining statements follow from  Propositions \ref{classical 1} and  Corollary \ref{change of coord Walg}.
\end{proof}

\section{A bihamiltonian structure on $\Q$}

In this section, we introduce an alternative set of coordinates $(s^1, \ldots, s^r)$ for $Q$, chosen to ensure that the  brackets of $\mathbb{B}_2^\mathcal{Q}$ are at most linear on $s^{r-1}(x)$. This enables us to define a local Poisson bracket $\mathbb{B}_1^\mathcal{Q}$ compatible with $\mathbb{B}_2^\mathcal{Q}$.

\bt \label{coordins}
Under the quasihomogeneous polynomial change of coordinates
\be\label{coords}
s^i =
\begin{cases}
z^i, & i \neq r, \\
\dfrac{r-1}{(r-1) + \alpha r} \big(z^r + \alpha z^1 z^{r-1}\big), & i = r,
\end{cases}
\ee
on $Q$, where $\alpha$ satisfies the equation
\be \label{quadratic eqn}
r \alpha^2 + 2(r-1) \alpha + (r-1) = 0,
\ee
the Poisson bracket $\mathbb{B}_2^\mathcal{Q}$ is at most linear in $s^{r-1}$. Furthermore, the following identities hold:
\[
\Omega_2^{11}(s) = r, \quad \Omega_2^{2i}(s) = i s, \quad \Gamma_{2,k}^{2i}(s) = \delta_k^i (i-1), \quad \partial_{s^{r-1}} \Omega_2^{i, r-i+1}(s) = r-1, \quad i = 1, \ldots, r.
\]
\et

\begin{proof}
Similar to  Lemma \ref{entries in z}, the coefficient of $(s^{r-1})^2$ in the entry $\O^{rr}_2(s)$ is proportional to the coefficient of $(z^{r-1})^2$ in the expression
\[
\Omega_2^{rr}(z) + 2 \alpha z^{r-1} \Omega_2^{r1}(z) + \alpha^2 (z^{r-1})^2 \Omega_2^{11}(z).
\]
This leads to the quadratic expression in equation \eqref{quadratic eqn}. The remainder of the theorem follows by applying the change of coordinates \eqref{coords} to the entries of $\mathbb{B}_2^\mathcal{Q}$.
\end{proof}

We fix the notation $(s^1, \ldots, s^r)$ for the coordinates introduced in Theorem \ref{coordins}. Observe that the discriminant of equation \eqref{quadratic eqn} is negative, yielding two complex conjugate values for $\alpha$. However, the results in this paper are independent of the specific choice of $\alpha$. The Poisson brackets corresponding to different values of $\alpha$ are related by taking the complex conjugate.

The following theorem provides a local Poisson bracket $\mathbb{B}_1^\mathcal{Q}$ on $\mathcal{Q}$ that is compatible with $\mathbb{B}_2^\mathcal{Q}$.

\bt \label{newbih}
The Lie derivative
\be
\mathbb{B}_1^\mathcal{Q} = \Lie_{\partial_{s^{r-1}(x)}} \mathbb{B}_2^\mathcal{Q}
\ee
defines a nontrivial local Poisson bracket on $\mathcal{Q}$ compatible with $\mathbb{B}_2^\mathcal{Q}$.
\et

\begin{proof}
From  Theorem \ref{coordins}, $\mathbb{B}_2^\mathcal{Q}$ is at most linear in $s^{r-1}(x)$ and explicitly depends on $s^{r-1}(x)$. Using  Corollary \ref{compatible poiss}, it follows that   $\Lie^2_{\partial_{s^{r-1}(x)}} \mathbb{B}_2^\mathcal{Q}=0$. Then Proposition \ref{serg} implies that $\mathbb{B}_2^\mathcal{Q}$ and $\mathbb{B}_1^\mathcal{Q}$ are compatible local Poisson brackets. 
\end{proof}

 \bx \label{gl3pb}
We verify the results for Lie algebra ${\mathfrak{gl}}_3$. For convenience, here and in the coming examples, we use superscripts for indices, we suppress the dependence on the independent variable $x$ and we write $\{u_i,u_j\}$ and $\delta$ for $\{u_i(x),u_j(y)\}$ and $\delta(x-y)$. 

The elements of $Q$ have the form
\beq 
Q=\left(
\begin{array}{ccc}
 \frac{1}{3}u_1 & 1 & 0 \\
 \frac{1}{2}u_2 & \frac{1}{3}u_1 & 1 \\
 u_3 & \frac{1}{2}u_2 & \frac{1}{3}u_1 \\
\end{array}
\right)
\eeq 
We write 
\beq L_2+b= \left(
\begin{array}{ccc}
 \frac{1}{3}b_1+\frac{1}{2}b_5+\frac{1}{2}b_6 & 1 & 0 \\
 \frac{1}{2}b_2+b_4 & \frac{1}{3}b_1-b_6 & 1 \\
 b_3 & \frac{1}{2}b_2-b_4 & \frac{1}{3}b_1-\frac{1}{2}b_5+\frac{1}{2}b_6 \\
\end{array}
\right)
\eeq
 and 
\beq \w=\left(
\begin{array}{ccc}
 0 & 0 & 0 \\
 \frac{1}{2}\w_2+\w_4 & 0 & 0 \\
 \w_3 & \frac{1}{2}\w_2-\w_4 & 0 \\
\end{array}
\right).
\eeq 
Then equation  \eqref{op:fixing1} leads to the solutions
\begin{eqnarray*}
\w_2 &=& b_5,~~\w_3= b_4-\frac{1}{2}b_6'+\frac{1}{2}b_5 b_6,~~\w_4= \frac{1}{2}b_6,\\\nonumber u_1&=& b_1,~u_2=
   b_2-b_5'+\frac{1}{4}b_5^2+\frac{3}{4} b_6^2,\\\nonumber u_3&=& b_3-b_4'-\frac{1}{4} b_6 b_5'-\frac{3}{4} b_5
   b_6'+\frac{1}{2}b_6''-\frac{1}{4}b_6^3+\frac{1}{4} b_5^2 b_6-\frac{1}{2}b_2 b_6+b_4 b_5.       
\end{eqnarray*}
The brackets of the local Poisson structure $\mathbb B$ on $\lop\g$ restricted to $b$ reads 
 \beq
\left(
\begin{array}{cccccc}
 3 \delta ' & 0 & 0 & 0 & 0 & 0 \\
 0 & 0 & 0 & b_3 \delta  & b_2 \delta  & 2 b_4 \delta  \\
 0 & 0 & 0 & 0 & 2 b_3 \delta  & 0 \\
 0 & -b_3 \delta  & 0 & 0 & b_4 \delta  & \frac{ 1}{2}b_2 \delta \\
 0 & -b_2 \delta  & -2 b_3 \delta  & -b_4 \delta  & 2 \delta ' & 0 \\
 0 & -2 b_4 \delta  & 0 & -\frac{1}{2}b_2 \delta  & 0 & \frac{2 }{3}\delta ' \\
\end{array}
\right)
 \eeq
 The nonzero brackets  of $\mathbb B_2^\Q$  are 
\begin{eqnarray}
 \{u_1 ,u_1\}_2^Q&=&3 \delta ' ,~~
     \{u_2 ,u_2\}_2^\Q=  -2 \delta ^{'''}+2 u_2 \delta '+u_2'\delta  ,~~
      \{u_2 ,u_3\}_2^\Q= 3 u_3 \delta '+2 u_3'\delta ,  \\\nonumber
\{u_3 ,u_3\}_2^\Q&=& \frac{1}{6}\delta ^{(5)}-\frac{5}{6}u_2 \delta ^{'''} -\frac{5}{4} u_2'\delta '' -\frac{3}{4} u_2''\delta '
   +\frac{2}{3} u_2^2 \delta '+\frac{2}{3}u_2 u_2' \delta  -\frac{1}{6}u_2{}^{'''} \delta.  
\end{eqnarray}
 We fix $\alpha=\frac{1}{3} \left(-2+i \sqrt{2}\right)$ as a solution of equation \eqref{quadratic eqn}. Then the coordinates $s_i$  are given by  
\[s_1=u_1,~~s_2=u_2+\frac{u_1^2}{6},~~s_3=-i \sqrt{2} u_3+\frac{1}{27} \left(3+2 i \sqrt{2}\right) u_1^3+\frac{2}{3} u_2 u_1.\]
The local Poisson brackets read
\begin{eqnarray}\label{Poiss gl3} 
 \left\{s_1 ,s_1 \right\}_2^\Q &=& 3 \delta ',~ \left\{s_1 ,s_2 \right\}_2^\Q= s_1 \delta '+\delta  s_1', \\\nonumber \left\{s_1 ,s_3 \right\}_2^\Q &=&(\frac{2}{3} i \sqrt{2} s_1^2 +\frac{2}{3} s_1^2 +2 s_2) \delta '+(\frac{4}{3} i \sqrt{2}   s_1 s_1'+\frac{4}{3}  s_1 s_1'+2  s_2')\delta, \\\nonumber 
 \left\{s_2 ,s_2 \right\}_2^\Q &=& -2 \delta ^{(3)}+2 s_2 \delta '+\delta  s_2' ,\\\nonumber
 \left\{s_2 ,s_3 \right\}_2^\Q&=& -\frac{4}{3} \delta ^{(3)} s_1-4 \delta '' s_1'-(4  s_1''-3
   s_3 )\delta '+(2   s_3'-\frac{4}{3}   s_1{}^{(3)})\delta , \\\nonumber 
  \left\{s_3 ,s_3 \right\}_2^\Q&=& -\frac{1}{3}\delta ^{(5)}+(\frac{5}{3}  s_2-\frac{7}{6}  s_1^2)\delta ^{(3)}+(\frac{5}{2} 
   s_2'-\frac{7}{2} s_1 s_1')\delta '' \\\nonumber & & +(\frac{3}{2}  s_2''-\frac{19}{6} s_1  s_1''-\frac{1}{3} s_1^4 +\frac{8}{9} i \sqrt{2} s_2 s_1^2 -\frac{4}{9} s_2 s_1^2 +4 s_3 s_1 -\frac{1}{2} \left(s_1'\right){}^2) \delta '\\\nonumber & & +\big(\frac{4}{9} i \sqrt{2}   s_1^2 s_2'-\frac{1}{3}   s_1' s_1''-\frac{2}{3}   s_1^3 s_1'-\frac{2}{9}  s_1^2
   s_2'+\frac{8}{9} i \sqrt{2}   s_2 s_1 s_1'\\\nonumber & ~~& -\frac{4}{9}   s_2 s_1 s_1'+2   s_1 s_3'+2   s_3 s_1'- s_1{}^{(3)} s_1+\frac{1}{3}   s_2{}^{(3)}\big) \delta. 
\end{eqnarray}
This confirms that $\mathbb B_2^\Q$ is at most linear in $s_{2}(x)$, thus justifying the construction of the compatible bracket $\mathbb B_1^\Q=\Lie_{\partial_{s_2(x)}}\mathbb B_2^\Q$. When we  take the complex conjugate of  $\alpha$, we get the complex conjugate of the local Poisson brackets. 

\ex

\section{Logarithmic Dubrovin-Frobenius manifolds} \label{Alg Frob mani}

In this section, we construct Dubrovin-Frobenius manifolds from the bihamiltonian structure $(\mathbb B_2^\Q,\mathbb B_1^\Q) $. We get a pair of matrices $(\O_2^{ij}(s), \O^{ij}_{1}(s))$ on $Q$, which arise from expanding the brackets in the form \eqref{loc poiss}. Note that $\O^{ij}_{1}(s)=\partial_{s^{r-1}} \O^{ij}_{2}(s)$.  In this section, we assume $r\geq 2$.

\bp \label{flat in z} The pair $(\O_2^{ij}(s), \O^{ij}_{1}(s))$  form a  flat pencil of metrics on  $Q$. There exists a quasihomogeneous polynomial change of coordinates of the form
\be  \label{ffflattt} t^1=s^1,~t^2= s^2, ~t^i=  s^i+\mathrm{non~ linear~ terms} \ee  such that  $\O_1^{ij}(t)=(r-1)\delta^{i+j,r+1}$ and  $\O_1^{ij}(t)=\partial_{s^{r-1}}\O_2^{ij}(t)$. Moreover, these coordinates preserve the identities \be \label{preserve the identities1}
\O_2^{11}(t)=r,~\O^{2j}_2(t)={j} t^j,  ~ \Gamma^{2j}_{2,k}(t)={(j-1)}\delta^j_k, ~j=1,\ldots,r.
\ee

\ep 
\begin{proof} From Proposition  \ref{coordins}, $\det \O^{ij}_1(s)\neq 0$.  Thus, the  matrices $\O_2^{ij}(s)$  is nondegenerate and, by applying Theorem \ref{DN thm} and the compatibility of the local Poisson brackets, the pair  $(\O_2^{ij}(s),\O_1^{ij}(s))$ defines a flat pencil of metrics on $Q$. Local flat coordinates
of the metric $\O_1^{ij}(s)$ exist at each point of $Q$  and can be found by equation \eqref{flatcond}.
 The proof of the existence of quasihomogeneous flat coordinates of the form 
\[t^i=s^i+\mathrm{non~ linear~ terms} , ~ i=1,\ldots,r\]
is given by corollary 2.4 in \cite{DCG} (see also Lemma 3.1 in \cite{WuZuo}). By Proposition \ref{change of coord Walg}, these coordinates $t^i$ preserve the identities \eqref{preserve the identities1}. 
Note that $t^{r}$ can not contains  a term proportional to  $s^1 
 s^{r-1}$, since otherwise $\O_1^{1r}(t)$ will depend on $t^{r-1}$ which will  break the quasihomogeneity property of the matrix $\O_1^{ij}(t)$. Therefore, $\partial_{t^{r-1}}=\partial_{s^{r-1}}$.  This ends the proof.  
\end{proof}

We fix the notation $(t^1,\ldots,t^r)$ for the flat coordinates constructed in Proposition~\ref{flat in z}. In this section, we use the notation $\partial_i := \partial_{t^i}$. According to \cite{DCG}, we have the following identities:
\be \label{FP1}
\partial_{k} \O_2^{ij} = \Gamma^{ij}_{2,k} + \Gamma^{ji}_{2,k}
\ee
and 
\be \label{FP2}
\O_2^{is} \Gamma^{jk}_{2,s} = \O_2^{js} \Gamma^{ik}_{2,s}.
\ee 
Moreover, there exist functions $f^j$ satisfying
\be \label{FP3}
\Gamma^{ik}_{2,s} = \eta^{im}\partial_{m}\partial_s f^k.
\ee 
Note that $\deg \Gamma^{jk}_{2,s}=j+k - s - 2$.

\begin{theorem}\label{FPonQ}
There exists a logarithmic Dubrovin-Frobenius manifold structure on the set
\[
Q\setminus \left(\{\det \O^{ij}_2=0\}\cup\{t^1=0\}\right),
\]
where the intersection form is $\O^{ij}_2$, the flat metric is $\O^{ij}_1$, the Euler vector field is 
\[
E=\frac{i}{r-1} t^i \partial_i,
\] 
and the identity vector field is $e=\partial_{r-1}$. Its potential has the following form 
\be \label{nonreg DF2}
\mathbb F(t^1,\ldots,t^r)=\frac{1}{(r-1)(2+4\delta_{r,3})}(t^{r-1})^2 t^2+\frac{1}{2(r-1)}\sum_{i \neq 2,r-1} t^i t^{r+1-i}+\frac{1}{2(r-1)}(t^r)^2\log t^r+G,
\ee
where $G$ is a quasihomogeneous polynomial in $(t^1,\ldots,t^{r-2},t^r)$ of degree $2r$. In addition,
\[
\Lie_E \mathbb F=\frac{2r}{r-1}\mathbb F+\frac{r}{2(r-1)^2}(t^r)^2.
\]
\end{theorem} 

\begin{proof}
The proof closely follows the arguments in \cite{WuZuo}, with the distinction that here we have $\O_2^{11}=r - 1$, whereas in \cite{WuZuo}, $\O_2^{11}=r(r - 1)$. Below we outline the essential steps.

From equation~\eqref{FP3}, we may assume without loss of generality that $f^j(t)$ is a quasihomogeneous polynomial of degree $j+r-1$. From equations \eqref{FP2} and \eqref{preserve the identities1} for $i=2$, we obtain 
\be \label{fps1}
(j-1)\,\O^{jk}_2=(j+k-2)\,\O_1^{jm}\partial_{m}f^k.
\ee 

We set
\be
F^j=\frac{r-1}{j-1}f^j,~j\neq 1;\quad D^i =\O_1^{im}\partial_m.
\ee 
Equation~\eqref{fps1} implies
\[
D^i F^j = D^j F^i,\quad i,j=2,\ldots,r.
\]

Moreover, the compatibility condition for the system of equations $D^k X=D^1 F^k$, $k\neq 1$, is $D^j D^k X=D^k D^j X$. This system can be solved uniquely up to a single-variable function in $t^r$. Hence, there exists a quasihomogeneous function $\mathbb F(t^1,\ldots,t^r,\log t^r)$ of degree $2r$, determined up to a single-variable term in $t^r$, such that
\be \label{df3}
F^k=D^k \mathbb F
\ee 
and
\be \label{df4}
\Lie_E \mathbb F=\frac{2r}{r-1}\mathbb F+\varphi(t^r),\quad E=\frac{i}{r-1}t^i\partial_i.
\ee 

Then, from \eqref{fps1}, it follows that 
\be \label{df0}
\O_2^{ji}=\Lie_E(\O_1^{jm}\O_1^{in}\partial_{n}\partial_{m}\mathbb F),\quad i\neq 1.
\ee 
To fix the quasihomogeneity uniquely, we require
\[
\O_2^{11}=\Lie_E(\O_1^{1m}\O_1^{1n}\partial_{n}\partial_{m}\mathbb F),
\]
which implies explicitly 
\be \label{df5}
\varphi(t^r)=\frac{r}{2(r-1)^2}(t^r)^2.
\ee

Let $\Pi_{ij}$ denote the inverse matrix of $\O_1^{ij}$. Using \eqref{fps1} for the cases $j=2$ and $j=1, k=2$, we find 
$\partial_i\partial_{r-1}\mathbb F=\Pi_{ij}t^j$. Hence,
\be \label{df6}
\partial_{r-1}\mathbb F=\frac{1}{2}\Pi_{ij}t^i t^j.
\ee 

Using these relations, we deduce that $\mathbb F$ has the explicit form given in~\eqref{nonreg DF2}. Next, define structure constants
\be
C_k^{ij}:=\O_1^{im}\O_1^{jn}\partial_m\partial_n\partial_k\mathbb F.
\ee

Then the following properties hold:
\[
C_k^{ij}=C_k^{ji},\quad C_{r-1}^{ij}=\O_1^{ij},\quad C_k^{ij}=\frac{r-1}{j-1}\Gamma_{2,k}^{ij}\quad (j\neq 1),
\]
and
\begin{eqnarray}
C_k^{i1}&=&\frac{r-1}{i-1}\Gamma_{2,k}^{1i},~ i\neq 1,\\[5pt]\nonumber
C^{11}_k&=&\frac{r-1}{t^r}\delta_{k}^1.
\end{eqnarray}

Detailed computations confirm that $C_k^{ij}$ define the structure constants of a Frobenius algebra on the cotangent space. In particular, these structure constants satisfy the WDVV equations:
\be
C_k^{ij} C_m^{kl}=C_k^{lj} C_m^{ki}.
\ee
\end{proof}

Finally, we note that the work by Arsie, Lorenzoni, Mencattini, and Moroni in \cite{Arsie} can alternatively be used to show that the tensor
\[
C_{ij}^k=\Pi_{im}C_j^{mk}
\]
defines a Dubrovin-Frobenius manifold structure. Their findings provide a valuable foundation for the subsequent developments in \cite{WuZuo}.

\bx \label{DFgl2}
 For the Lie algebra  $\mathfrak{gl}_2$, the nonzero brackets of $\mathbb B_2^\Q$ are 
\[
 \left\{u_1,u_1\right\}_2^\Q= 2 \delta ',~~ \left\{u_2,u_2\right\}_1^\Q= -\frac{1}{2}\delta ^{(3)}+2 u_2 \delta '+ u_2'\delta \]
 Using the coordinates 
 \[z_1=u_1,~~z_2=u_2+\frac{1}{4}u_1^2\]
leads to the brackets 
\begin{eqnarray*}
 \left\{z_1,z_1\right\}_2^Q&=& 2 \delta ' ,~~ \left\{z_1,z_2\right\}_2^\Q= z_1 \delta
   '+z_1'\delta   \\\nonumber
 \left\{z_2,z_2\right\}_2^\Q&=& -\frac{1}{2}\delta ^{(3)}+2 z_2 \delta '+z_2'\delta 
\end{eqnarray*}
It is  almost linear in $z_1$. Then $\mathbb B_1^Q=\Lie_{\partial_{z^1}}\mathbb B_2^Q$ defines a compatible Poisson bracket. Moreover, $(t_1,t_2)=(z_1,z_2)$ are the flat coordinates of $\O_1^{ij}$ and the corresponding Dubrovin-Frobenius manifold has the potential 
\be\label{frob2}\mathbb F=\frac{1}{2} t_2 t_1^2+\frac{1}{2} t_2^2 \log \left(t_2\right).\ee
\ex
\bx \label{DFgl3}
We consider the Lie algebra $\mathfrak{gl}_3$. From local Poisson brackets  \eqref{Poiss gl3} given in Example \ref{gl3pb}, we get  
\[\O_2^{ij}(s)=\left(
\begin{array}{ccc}
 3 & s_1 & \frac{2i \sqrt{2}}{3}  s_1^2+\frac{2}{3} s_1^2+2 s_2 \\
 s_1 & 2 s_2 & 3 s_3 \\
 \frac{2i \sqrt{2}}{3}  s_1^2+\frac{2}{3} s_1^2+2 s_2 & 3 s_3 & -\frac{s_1^4}{3}+\frac{8i \sqrt{2}}{9}  s_2 s_1^2-\frac{4}{9} s_2 s_1^2+4 s_3 s_1 \\
\end{array}
\right)\]
In the flat coordinates of $\O_1^{ij}=\partial_{s^{r-1}}\O_2^{ij}$,
\[ t_1=s_1,~~t_2=s_2,~~t_3=s_3+\frac{1}{27} \left(1-2 i \sqrt{2}\right) s_1^3,\]
we have
\[ \O_2^{ij}(t)=\left(
\begin{array}{ccc}
 3 & t_1 & t_1^2+2 t_2 \\
 t_1 & 2 t_2 & 3 t_3 \\
 t_1^2+2 t_2 & 3 t_3 & 4 t_1 t_3 \\
\end{array}
\right) \]
 To find the potential of the corresponding Dubrovin-Frobenius manifold,  we set 
\[\mathbb F(t)=F_1\left(t_1,t_3\right)+F_2(t_3)+\frac{t_2^3}{12}+\frac{1}{2} t_1 t_3 t_2.\]
 Then  the definition of the intersection form $\O_2^{ij}(t)$ will give  partial differential equations for $F_1\left(t_1,t_3\right)$ (see \cite{DFP} for details). While the WDVV equations will lead to a differential equation for $F_2(t_3)$. Solving these equations leads to the potential
\be \label{frob3}\mathbb F=\frac{1}{12} t_3 t_1^3+\frac{1}{2} t_2 t_3 t_1+\frac{t_2^3}{12}+\frac{1}{4} t_3^2 \log
   \left(t_3\right).\ee
The complex conjugate of the entries of $\O_2^{ij}(s)$ gives the same potential. 
\ex 
\bp\label{firsagd}
The Lie derivative $\Lie_{\partial_{s^r(x)}} \mathbb{B}_2^\mathcal{Q}$ defines a local Poisson bracket (the first Adler-Gelfand-Dickey bracket) that is compatible with $\mathbb{B}_2^\mathcal{Q}$. However, the leading term of the bihamiltonian structure $(\mathbb{B}_2^\mathcal{Q}, \Lie_{\partial_{s^r}} \mathbb{B}_2^\mathcal{Q})$ does not define a flat pencil of metrics.
\ep

\begin{proof}
Similar to Proposition \ref{exponent in u}, using $\deg s^r=r$, we can show that $\mathbb{B}_2^\mathcal{Q}$ is almost linear in $s^r$. Hence, by Proposition \ref{serg}, the Lie derivative $\Lie_{\partial_{s^r(x)}} \mathbb{B}_2^\mathcal{Q}$ defines a Poisson bracket compatible with $\mathbb{B}_2^\mathcal{Q}$. 
Due to quasihomogeneity, the entries of the first row of the matrix $\partial_{s^r} \Omega_2^{ij}(s)$ equal zero. Therefore, the matrix $\partial_{s^r}\O_1^{ij}$ degenerate and the leading term of the bihamiltonian structure $(\mathbb{B}_2^\mathcal{Q}, \Lie_{\partial_{s^r(x)}} \mathbb{B}_2^\mathcal{Q})$ fails to define a flat pencil of metrics.
\end{proof}

\section{Relation to invariant theory}\label{InvTheo}

In this section, we demonstrate that the logarithmic Dubrovin-Frobenius manifolds constructed in this article can also be realized on the orbits space of the standard representation of the permutation group $S_r$. Here, the representation is given by permuting the coordinates of an $r$-dimensional complex vector space.

Let $\psi: G \to GL(V)$ be a linear representation of a finite group $G$ on a complex vector space $V$. The ring of invariant polynomials ${\mathbb C}[\psi]$ associated with this representation is finitely generated by homogeneous polynomials, and it is the coordinate ring of the orbits space variety ${\mathcal O}(\psi)=V/G$ arising from the group action of $G$ on $V$ (see \cite{dresken}). Let $(p^1,\ldots,p^n)$ be linear coordinates on $V$. Then, given any invariant polynomial $f\in {\mathbb C}[\psi]$, the Hessian 
\[
\mathrm H(f):=\frac{\partial^2 f}{\partial p^i \partial p^j}
\]
defines a bilinear form on the tangent spaces on the orbits space ${\mathcal O}(\psi)$ (for details, see \cite{Orlik}). 

The Dubrovin–Saito method provides a general approach for constructing Dubrovin-Frobenius manifolds via invariant theory for linear representations of finite groups. This construction was pioneered by Dubrovin in \cite{DCG}, where he built polynomial Dubrovin-Frobenius manifolds on the orbit spaces of reflection groups. Dubrovin's work was inspired by K. Saito's construction of flat coordinates on these orbits spaces \cite{Saito}. This method eventually led to the classification of a particular class of polynomial Dubrovin-Frobenius manifolds up to equivalence \cite{HER}.

The Dubrovin–Saito construction can be summarized as follows (see \cite{Almamari2} for more details). Let $\psi: G \to GL(V)$ be a linear representation. To construct a Dubrovin–Frobenius manifold structure from ${\mathbb C}[\psi]$:

\begin{enumerate}
    \item Fix a homogeneous invariant polynomial $f$.
    \item Verify that the inverse of the Hessian $\mathrm H(f)^{-1}$ defines a contravariant flat metric $\overline{\O}_2^{ij}$ on some open subset $U\subseteq {\mathcal O}(\psi)$.
    \item Construct another contravariant metric $\overline{\Omega}_1^{ij}$, such that $(\overline{\Omega}_2^{ij},\overline{\Omega}_1^{ij})$ form a flat pencil of metrics.
    \item Verify that the resulting flat pencil of metrics corresponds to a Dubrovin-Frobenius manifold structure on an open subset of $U$.
\end{enumerate}

A Dubrovin-Frobenius manifold structure obtained through Dubrovin-Saito method will be called a \textit{natural Frobenius manifold structure} on the orbits space.

\bt \label{FBonOrb}
The orbits space of the standard representation of the permutation group $S_r$ carries a natural structure of logarithmic Dubrovin-Frobenius manifold locally biholomorphic to Dubrovin-Frobenius manifolds given by Theorem \ref{FPonQ}.
\et
\begin{proof}
     We will use the fact that  the standard representation of the permutation group $S_r$  is isomorphic to the standard representation of the  Weyl group of $\g$  on its Cartan subalgebra. We fix Cartan subalgebra $\h= \gr_0=\g^h$. Let $\T$ be the space of operators of the
form \beq Y=\partial_x+ p + L_2,~~p\in \lop{\h}.\eeq  We write the elements of $\T$ in the form 
 \[\partial_x+p+L_2= \partial_x+L_2+\sum_{i=1}^{r} p^{i}(x) \e_{i,i}.\]
 Then the restriction of the Poisson bracket \eqref{Poissbrac} to $\T$ defines a local Poisson bracket admitting a dispersionless limit, i.e., \be \label{Poison cartan}\{p^i(x),p^j(y)\}^\T=\delta^{ij}\delta'(x-y).\ee  
 We consider $\T$ as a  subspace of $\B$ in the gauge action \eqref{gauge action}. This leads to  Miura  transformation \be
\Phi:\T \to \Q\ee
defined by sending an operator $Y$ to its
conjugacy class in $\Q$. As a consequence,  the densities $z^i(x)$  (see \eqref{norm-coord eq}) can be written as  differential polynomials in  $p^j(x)$. Moreover, their  non-differential parts 
\be \label{InvMius}\overline z^i(x):=\dfrac{1}{i}\sum_{j=1}^r (p^j(x))^i.\ee
correspond to the power-sum symmetric polynomials which form a complete set of  generators of the  invariant ring of the standard action  of the permutation  group $S_r$ on the coordinates $p^{i}$ by permuting the indices \cite{dresken}. 

By Proposition 3.26 in \cite{DS},   the map  $\Phi$  is a Hamiltonian map into the Poisson bracket $\mathbb B_2^\Q$.  Thus, we can obtain the brackets of $\mathbb B_2^\Q$  by using Leibniz rule on the Poisson bracket \eqref{Poison cartan}
\be\label{anay:brac} \{z^i(x),z^j(y)\}_{2}^\Q=
  \sum_{k, n\geq \beta;m\geq \alpha} (-1)^{n}{m \choose \alpha}{n\choose \beta}{\partial z^i(x)\over \partial (p^k)^{(m)}} \left({\partial z^j(x)\over \partial (p^k)^{(n)}}\right)^{(\alpha +\beta)}  \delta^{m+n-\alpha-\beta+1}(x-y).
\ee
As $\O_2^{ij}(z)$ can be considered as a metric on $Q$, it follows from the algebraic independence that its entries are uniquely determined by the non-differential part of \eqref{anay:brac}, i.e., by setting $m=n=\alpha=\beta=0$. Thus, 
 \beq \O_2^{ij}(z)=\O_2^{ij}(\overline z)=\sum_k{\partial \overline z^i\over \partial p^k}{\partial \overline z^j\over \partial p^k}.\eeq
 Thus, $\O^{ij}_2(z)$ is identical to the metric defined on the orbits space  by the inverse of the Hessian of $\overline z^2$. Then, we  utilize Theorem \ref{coordins} and Theorem  \ref{FPonQ} to obtain the same logarithmic Dubrovin-Frobenius manifold structure on an  open dense subset of $\mathcal{O}(\psi)$. In particular, under the change of coordinates $(s^1\ldots,s^r)$ of the form \eqref{coords}, the matrix  $\O^{ij}_2(s)$ is linear in $s^{r-1}$ and the Lie derivative $\O_1^{ij}=\Lie_{\partial_{s^{r-1}}}\O^{ij}_2$ define a flat pencil of metrics leading to a logarithmic Dubrovin-Frobenius manifold structure given by Theorem \ref{FPonQ}.  
\end{proof}
 
Theorem \ref{FBonOrb} gives a shortcut to construct the flat pencil of metrics $(\O_2^{ij},\O_1^{ij})$ without constructing the entire local bihamiltonian structure $(\mathbb B_2^\Q,\mathbb B_1^\Q) $ as illustrated in the following example. 

\bx\label{DFgl4}
We illustrate the construction of the logarithmic Dubrovin Frobenius manifold for the Lie algebra  $\mathfrak{gl}_4$  using the Dubrovin-Saito method. From the the invariant \eqref{InvMius}, we get 
\[\O_2^{ij}(z)=\left(
\begin{array}{cccc}
 4 & z_1 & 2 z_2 & 3 z_3 \\
 * & 2 z_2 & 3 z_3 & 4 z_4 \\
 * & * & 4 z_4 & -\frac{z_1^5}{24}+\frac{5}{6} z_2 z_1^3-\frac{5}{2} z_3 z_1^2-\frac{5}{2} z_2^2 z_1+5 z_4 z_1+5 z_2 z_3 \\
* & * & * &
   -\frac{z_1^6}{24}+\frac{3}{4} z_2 z_1^4-2 z_3 z_1^3-\frac{3}{2} z_2^2 z_1^2+3 z_4 z_1^2-z_2^3+3 z_3^2+6 z_2 z_4 \\
\end{array}
\right).\]
We fix $\alpha=\frac{1}{4} \left(-3+i \sqrt{3}\right)$ as a solution of equation \eqref{quadratic eqn}. Then, under the  change of coordinates  \eqref{coordins}
\[s_1=z_1,~s_2=z_2,~s_3=z_3,~s_4=-i \sqrt{3} z_4+\frac{3i \sqrt{3}}{4}  z_1 z_3+\frac{3}{4} z_1 z_3,\]  the matrix $\O_2^{ij}(s)$ is linear in $s_{r-1}$. The  flat coordinates for $\O_1^{ij}(s)$  are 
\begin{eqnarray*}
    t_1&=&s_1,~~t_2=\dfrac{1}{3} s_2,~ t_3=s_3+\frac{1}{72} i \left(\sqrt{3}-3 i\right) s_1^3+\frac{1}{6} \left(-3+i \sqrt{3}\right) s_2 s_1,\\
    t_4&=& s_4+\frac{1}{32} \left(1+i
   \sqrt{3}\right) s_1^4+\frac{1}{8} \left(-3-i \sqrt{3}\right) s_2 s_1^2+\frac{1}{2} i \sqrt{3} s_2^2.
\end{eqnarray*}
Thus,
   \[\O_2^{ij}(t)=\left(
\begin{array}{cccc}
 4 & t_1 & \frac{i}{\sqrt{3}}( t_1^2+2  t_2)& \frac{i}{\sqrt{3}} t_1^3+i \sqrt{3} t_2 t_1+3 t_3 \\
 * & 2 t_2 & 3 t_3 & 4 t_4 \\
 * & * & \frac{2}{3} t_2^2+\frac{4 i}{\sqrt{3}} t_4 & \frac{5 i}{\sqrt{3}} t_1 t_4 \\
 * & * & * & 2 i \sqrt{3} t_4 t_1^2+2 i \sqrt{3} t_2 t_4 \\
\end{array}
\right).\]
The potential of the  associated Dubrovin-Frobenius manifold structure 
\be \label{frob4} \mathbb F=\frac{1}{3} t_3 t_4 t_1+\frac{1}{6} t_2 t_3^2+\frac{i}{36 \sqrt{3}} t_4 t_1^4+\frac{i}{6 \sqrt{3}} t_2 t_4 t_1^2+\frac{1}{216}t_2^4+\frac{i}{6 \sqrt{3}} t_2^2
   t_4+\frac{1}{6} t_4^2 \log t_4.\ee
Replacing $\alpha$ by its complex conjugate yields the complex conjugate of the potential $\mathbb F$. 
\ex

\section{Conclusion and Remarks}

In this work we constructed a  bihamiltonian structure beginning from the unconstrained Gelfand-Dickey-Adler local Poisson bracket. This bihamiltonian structure admits a dispersionless limit and its leading term defines a logarithmic Dubrovin-Frobenius manifolds structure.  Recall that such a bihamiltonian structure, and hence the associated Dubrovin–Frobenius manifold structure, is called semisimple if the roots $a^1,\ldots,a^r$ of the characteristic polynomial
\be 
\Psi(\lambda;u):=\det (\Omega^{uv}_2(u)-\lambda \Omega^{uv}_1(u))
\ee 
are pairwise distinct at some points. In this case,  $(a^1,\ldots,a^r)$ define local coordinates. Moreover, writing the higher-order terms of the bihamiltonian structure as 
\be 
 \{u^i(x),u^j(y)\}^{[k]}_\alpha  =  S^{ij}_{2,k}(u(x)) \delta^{k+1}(x-y)+\ldots,~  ~k>0,~~\alpha=1,2
\ee
we can calculate the central invariants of the bihamiltonian structure, under the assumption $ \{u^i(x),u^j(y)\}^{[1]}_\alpha=0$, by the formulas \cite{DLZ}
\be \label{cent form}
c_i(a^i):=\frac{1}{3}[\frac{d\Psi}{d\lambda}(a^i;u)]^2\frac{\frac{\partial \Psi(\lambda;u)}{\partial u^k}\frac{\partial \Psi(\lambda;u)}{\partial u^l}(S_{2;2}^{kl}(u)-\lambda S_{1;2}^{kl}(u))}{[\frac{\partial \Psi(\lambda;u)}{\partial u^k}\frac{\partial \Psi(\lambda;u)}{\partial u^l}\Omega_1^{kl}(u)]^2}{\Big |}_{\lambda=a^i}.
\ee 
We refer the reader to \cite{DLZ} for the definition and details on the role of central invariants in classifying semisimple bihamiltonian structures under Miura transformations. 

We recall that if the central invariants are all equal and constant, the bihamiltonian structure is of topological type, meaning it  can be reconstructed using identities inspired by the theory of Gromov-Witten invariants \cite{DZ}.

We confirm that the bihamiltonian structures constructed for $r=2,3,4$ are of topological type. The central invariants equal $-\frac{1}{24}$ for $r=2$ and $-\frac{1}{8}$ for $r=3,4$. This strongly indicates that the bihamiltonian structure for arbitrary $r$ might also be of topological type. 

Let $M$ and $\widetilde M$ be two Frobenius manifolds  with flat metrics $\Pi$ and $\widetilde \Pi$ and potentials $\mathbb F$ and $\widetilde{\mathbb F}$, respectively. We say $M$ and $\widetilde{M}$ are locally equivalent if there are open sets $U\subseteq M$ and $\widetilde U \subseteq \widetilde M $ with a local diffeomorphism $\phi:U\to \widetilde U$ such that 
\beq\label{equivalent} \phi^* \widetilde\Pi=c \Pi,\eeq 
for some nonzero constant $c$, and $\phi_*:T_uU\to T_{\phi(u)} \widetilde U,~u\in U$ is an isomorphism of Frobenius algebras \cite{DuRev}. Note that, in this case, it is not necessary that $\phi^* \widetilde{\mathbb F}=\mathbb F$. In coordinates, this means the structure constants are related by 
\[
C_{ij}^{m} \;=\; 
\sum_{p,q,\ell=1}^n 
\frac{\partial \widetilde{t}^{p}}{\partial t^{i}} 
\;\frac{\partial \widetilde{t}^{q}}{\partial t^{j}} 
\;\widetilde{C}^{\ell}_{pq}
\;\frac{\partial t^{m}}{\partial \widetilde{t}^{\ell}}\,.
\]
Note that two bihamiltonian structures of topological type are equivalent, if the corresponding Dubrovin-Frobenius manifolds are equivalent. From the theory of central invariants, this means that we can transform one to the other by using  Miura transformation, a transformation of the form  
\begin{equation}\label{miura}
u^i\mapsto F^i_0(u)+\sum_{k\ge 1}  F_k^i(u,u_x,\cdots,u^{(k)})
\end{equation}
where $F^i_k$ is homogeneous differential polynomial  with $\deg F^i_k=k$ and the map $u^i \mapsto F^i_0(u)$ is locally biholomorphic.
 
Liu, Zhang and Zhou  \cite{LiZhang} introduced   a bihamiltonian structure associated to the constrained KP hierarchy. They proved that its leading term produces a logarithmic Dubrovin–Frobenius manifold and it is of topological type. This bihamiltonian structure is defined as follows.  We consider pseudodifferential operators of the form 
\begin{equation}\label{zh-2}
\L = D^{n+1} + v^n D^{n-1} + \cdots + v_2D + v_1 + (D-v^{n+1})^{-1} v^{n+2}.
\end{equation}
Here,  the variational derivative of a functional $F$ is defined as
\begin{align}\label{zh-3}
\frac{\delta F}{\delta \L} := \sum_{i=1}^n D^{-i} \frac{\delta F}{\delta v^i}+ \frac{\delta F}{\delta v^{n+2}} +\frac{\delta F}{\delta v^{n+1}} \frac{1}{v^{n+2}} (D- v^{n+1}), 
\end{align}
Denote the variational derivatives of two functionals $F$ and $G$ by $X$ and $Y$, respectively. Then the two compatible Poisson brackets are 
\begin{align}
&\{F,G\}_1^{kp} = \int \mathrm{res} \left([\L,X_+] Y-[\L,X]_+Y\right) \; dx ,\label{ds-bh-1}\\
&\{F,G\}_2^{kp} = \int \mathrm{res} \left( (\L Y)_+ \L X - (Y \L)_+ X \L + \frac{1}{n+1} X[\L,K_Y] \right) \; dx.\label{ds-bh-2} 
\end{align}
Here $K_Y$ is given by the differential polynomial $\partial^{-1}_x \mathrm{res}([\L,Y])$.

 Ma and Zuo \cite{MaZuo} demonstrated that the obtained  logarithmic Dubrovin-Frobenius manifold from the leading terms of $(\{\cdot,\cdot\}_1^{kp} ,\{\cdot,\cdot\}_1^{kp} )$ is isomorphic to one constructed in \cite{Arsie} on the orbits space of reflection group of type \( B_r \) using Dubrovin-Saito method outlined in section \ref{InvTheo}. Note that they  can also be constructed  on the orbits space of the standard representation of  the permutation group  beginning from the invariant metric \cite{WuZuo}
\[<dp^i,dp^j>=1-\delta^{ij}.\]
In dimensions 3 and 4, These structures are represented by the potentials \cite{Arsie} 
\begin{align*}
    F_{B_3}&=\frac{1}{12} \widetilde{t}_3 \widetilde{t}_1^3+\widetilde{t}_2 \widetilde{t}_3 \widetilde{t}_1+\frac{\widetilde{t}_2^3}{6}+\widetilde{t}_3^2 \log \widetilde{t}_3,\\
    F_{B_4}&=\frac{1}{108} \widetilde{t}_4 \widetilde{t}_1^4+\frac{1}{6} \widetilde{t}_2 \widetilde{t}_4 \widetilde{t}_1^2+\widetilde{t}_3 \widetilde{t}_4 \widetilde{t}_1-\frac{\widetilde{t}_2^4}{72}+\frac{1}{2} \widetilde{t}_2 \widetilde{t}_3^2+\frac{1}{2} \widetilde{t}_2^2 \widetilde{t}_4+\frac{3}{2} \widetilde{t}_4^2 \log \widetilde{t}_4.
\end{align*}
We confirm that they are equivalent to those given by the potentials in Examples  \ref{DFgl3} and \ref{DFgl4}, respectively,  under the maps 
   \[\widetilde{t}_1=\sqrt{2} t_1,~ \widetilde{t}_2=t_2,~ \widetilde{t}_3=\frac{1}{\sqrt{2}}t_3,\]
 \[\widetilde{t}_1= -i \sqrt{3}t_1,~\widetilde{t}_2= -t_2,~\widetilde{t}_3= t_3,~\widetilde{t}_4= -\frac{i}{\sqrt{3}}t_4.\]
From the above calculations and discussion, we conclude that for \( r =  3, 4 \), the bihamiltonian structures \( (\mathbb{B}_2^\mathcal{Q}, \mathbb{B}_1^\mathcal{Q}) \) are equivalent to the bihamiltonian structures $(\{\cdot,\cdot\}_2^{kp} ,\{\cdot,\cdot\}_1^{kp} )$ associated with the constrained KP hierarchies. We conjecture that this equivalence extend to all cases with $r>2$.

\vspace{0.1cm}

\noindent{\bf Acknowledgments.}
The author would like to thank the Isaac Newton Institute for Mathematical Sciences for support and hospitality during the programme Dispersive Hydrodynamics when work on this paper was initiated (EPSRC Grant Number EP/R014604/1).
\vspace{0.1cm}
 

\vspace{0.1cm}

\noindent{\bf Data Availability} Non applicable. 
\section*{Declarations}

\noindent{\bf Conflict of interest} The authors have no relevant financial or non-financial interests to disclose.

~~~~~~~~
~~~~~~~~

\noindent \textbf{Visiting Associate Professor of Mathematics} \\
\noindent Division of Science \\
\noindent New York University Abu Dhabi \\
\noindent \texttt{yd3157@nyu.edu} \\

\noindent \textbf{Associate Professor of Mathematics} \\
\noindent Department of Mathematics, College of Science \\
\noindent Sultan Qaboos University \\
\noindent \texttt{dinar@squ.edu.om} \\



\begin{thebibliography}{99}
\bibitem{Almamari2}  Al-Maamari, Z., Dinar, Y.; Frobenius manifolds on orbits spaces, Math. Phys. Anal. Geom. 25 , no. 3, Paper No. 22, 26 (2022).
 \bibitem{Arsie} Arsie, A., Lorenzoni, P., Mencattini, I., Moroni, G.; A Dubrovin-Frobenius manifold structure of NLS type on the orbit space of $B_n$, 
Sel. Math. New Ser. 29, 2 (2023). 


\bibitem{gDSh2}  Burroughs, N., de Groot, M., Hollowood, T. and Miramontes, J.;
 Generalized Drinfeld-Sokolov hierarchies II: the Hamiltonian
 structures, Comm. Math. Phys.153, 187 (1993).
 \bibitem{dickeylec} Dickey, L. A.;
Lectures on classical W-algebras, Acta Appl. Math.47, no.3, 243–321 (1997).
 \bibitem{mypaper} Dinar, Yassir, On classification and construction of algebraic Frobenius manifolds. Journal of Geometry and Physics, Volume 58, Issue 9,1171-1185, September (2008).
\bibitem{wdvv} Dijkgraaf, R., Verlinde, H., Verlinde, E.; Topological strings in $d  1$.
Nucl. Phys. B352, 59 (1991).
 \bibitem{mypaper1} Dinar, Y.; Frobenius manifolds from regular classical W-algebras. Advances in Mathematics, Volume 226,
Issue 6, Pages 5018-5040 (2011).
\bibitem{mypaper4} Dinar, Y.; $W$-algebras and the equivalence of bihamiltonian, Drinfeld-Sokolov and Dirac reductions. J. Geom. Phys. 84, 30–42 (2014).
\bibitem{mypaper6} Dinar, Yassir; Algebraic classical W-algebras and Frobenius manifolds. Lett Math Phys 111, 115 (2021).
\bibitem{rank 3} Dinar, Y.;
Low-dimensional bihamiltonian structures of topological type. J. Math. Phys. 64, no.3, Paper No. 033502, (2023).
\bibitem{dresken} Derksen, H.; Kemper, G., {\it Computational Invariant Theory,} Springer 2015.
\bibitem{DS} Drinfeld, V. G., Sokolov, V. V.; Lie algebras and equations of Korteweg-de Vries type. (Russian)  Current problems in mathematics, Vol. 24,  81-180, Itogi Nauki i Tekhniki, Akad. Nauk SSSR, Vsesoyuz. Inst. Nauchn. i Tekhn. Inform., Moscow, (1984).

\bibitem{DN} Dubrovin, B. A., Novikov, S. P.; Poisson brackets of hydrodynamic type. (Russian)  Dokl. Akad. Nauk SSSR  279,  no. 2, 294-297  (1984).
 \bibitem{DuRev} Dubrovin, B.; Geometry of $2$D topological field theories.  Integrable systems and quantum groups (Montecatini Terme, 1993),  120-348, Lecture Notes in Math., 1620, Springer, Berlin, (1996).
\bibitem{DCG} Dubrovin, B.; Differential geometry of the space of orbits of a Coxeter group.  Surveys in differential geometry IV: integrable systems,  181-211 (1998).
\bibitem{DFP} Dubrovin, B.; Flat pencils of metrics and Frobenius manifolds.  Integrable systems and algebraic geometry (Kobe/Kyoto, 1997),  47-72, World Sci. Publ. (1998).
\bibitem{DZ} Dubrovin, B. , Zhang, Y.; Normal forms of hierarchies of integrable PDEs, Frobenius
manifolds and Gromov-Witten invariants, www.arxiv.org
math/0108160.

\bibitem{DLZ} Dubrovin, B. , Liu, Si-Qi,  Zhang, Y.; Frobenius manifolds and central invariants for the Drinfeld–Sokolov bihamiltonian structures, Advances in Mathematics, Volume 219, Issue 3,  780-837 (2008).
 \bibitem{Elash} Elashvili, A. G., Kac, V. G., Vinberg, E. B.; Cyclic elements in semisimple Lie algebras. Transform. Groups 18, no. 1,97–130 (2013).
  \bibitem{BalFeh1} Feher, L., O'Raifeartaigh, L., Ruelle, P., Tsutsui, I., Wipf, A.; On Hamiltonian reductions of the Wess-Zumino-Novikov-Witten theories.  Phys. Rep.  222, no. 1 (1992).

\bibitem{fehercomp} Feher, L., O'Raifeartaigh, L., Ruelle, P., Tsutsui, I.; On the completeness of the set of classical $ W$-algebras obtained from DS reductions.  Comm. Math. Phys.  162 ,  no. 2, 399-431 (1994).

\bibitem{HER} Hertling, C.; Frobenius manifolds and moduli spaces for singularities. Cambridge Tracts in Mathematics, 151. Cambridge University Press, ISBN: 0-521-81296-8 (2002).
\bibitem{Humph} Humphreys, J. E.;
Introduction to Lie algebras and representation theory. Grad. Texts in Math., 9
Springer-Verlag, New York-Berlin, ISBN:0-387-90053-5 (1978).
\bibitem{LiZhang} Liu, S.; Zhang, Y., Zhou, X.;
Central invariants of the constrained KP hierarchies. 
J. Geom. Phys.   97 , 177–189 (2015).
138, pp. 154-167 (2019).
\bibitem{MaZuo} Ma, S., Zuo, D.;
Dubrovin-Frobenius manifolds associated with $B_n$ and the constrained KP hierarchy. J. Math. Phys.   64 , no. 5 (2023).

 \bibitem{Orlik} Orlik, P.; Solomon, L., {\it The Hessian Map in The Invariant Theory of Reflection Groups,} Nagoya Math. J. {\bf 109} (1988) 1–21 . 
\bibitem{Saito} Saito, K., Yano, T.,  Sekeguchi J., \textit{On a certain generator system of the ring of
invariants of a finite reflection group}, Comm. in Algebra 8(4) (1980).

\bibitem{serg} Sergyeyev, A.; A Simple Way of Making a Hamiltonian System into a Bi-Hamiltonian One, Acta Applicandae Mathematica, \textbf{83}(1), 183-197 (2004)
  \bibitem{sldwy2} Slodowy P.; Four lectures on simple groups and singularities, Communications of the
Math. inst. Rijksun. Utrecht, 11 (1980).

\bibitem{wang}  Wang, W.; Nilpotent orbits and finite W-algebras. Geometric representation theory and extended affine Lie algebras, 71–105, Fields Inst. Commun., 59, Amer. Math. Soc.,  (2011).
\bibitem{WuZuo} Wu, Y., Zuo, D.; Dubrovin–Frobenius manifold structures on the orbit space of the symmetric group. J. Math. Phys.65 (1), 011702 (2024).
  
 



\end{thebibliography}
\end{document}